\documentclass[12pt]{amsart}
\usepackage{graphicx, graphics, amssymb, amsmath, amscd,latexsym}
\usepackage[abs]{overpic} % この行を追加
\newtheorem{thm}{Theorem}[section]
\newtheorem{lem}[thm]{Lemma}
\newtheorem{conj}[thm]{Conjecture}

\newtheorem{rmk}[thm]{Remark}

\newtheorem{dummy}{Theorem}
\newtheorem{dummmy}{Conjecture}
\newtheorem{dummmmy}{Lemma}

\def\Z{\mathbb Z}

\def\R{\mathbb R}

\newcommand{\nbf}[1]{\noindent \textbf{#1}}

\begin{document}
\title[A construction of slice knots via annulus twists]
{A construction of slice knots\\ via annulus twists}
\author{Tetsuya Abe and Motoo Tange}
\address{Department of Mathematics,
Tokyo Institute of Technology,
2-12-1 Ookayama, Meguro-ku, 
Tokyo 152-8551, Japan}
\email{abe.t.av@m.titech.ac.jp}
\address{Institute of Mathematics, 
University of Tsukuba, Ibaraki 305-8571, Japan}
\email{tange@math.tsukuba.ac.jp}
\date{\today}
\maketitle
%%%%%%%%%%%%%%%%%%%%%%%%%%%%%%%%%%%%%%
\begin{abstract}
We give a new construction of slice knots via annulus twists.
The simplest slice knots obtained by our method 
are those  constructed by Omae.
In this paper,
we introduce a sufficient condition for given  slice knots to be ribbon,
and prove that  all Omae's knots are  ribbon.
%It is not clear whether other slice knots obtained by our construction
%are ribbon or not.
\end{abstract}
\section{Introduction}
The annulus twist is a certain operation on knots 
along an  annulus embedded in the 3-sphere $S^3$.
Osoinach  \cite{Os} found that this operation is useful in the study of 3-manifolds.
Using annulus twists, he gave 
the first example of a 3-manifold 
admitting infinitely many presentations by $0$-framed knots.
For more studies, 
see  \cite{{AJOT}, {AJLO}, {BGL}, {K}, {Tak}, {Te},  {Om}}.

Recently, the first author, Jong, Omae and Takeuchi \cite{AJOT}
constructed  knots related to the slice-ribbon conjecture:
Let $K$ be a slice knot admitting an annulus presentation (for the definition, see Section 2)
and  $K_n$  $(n \in \Z)$ the knot obtained from $K$ by the $n$-fold annulus twist.
They proved that $K_n$
bounds a smoothly embedded disk in a certain homotopy 4-ball $W(K_n)$ with $\partial W(K_n) \approx S^3$.
A natural question is the following:\\
\vskip -2mm

\noindent \textbf{Question.}
Is $W(K_n)$ diffeomorphic to the standard 4-ball $B^4$?\\
\vskip -2mm

If  $W(K_n)$ is not diffeomorphic to $B^4$,
then the homotopy 4-sphere obtained 
by capping it off 
% (with a $4$-handle) 
is a counterexample of the smooth 4-dimensional Poincar\'e conjecture.
For related studies,
see \cite{{A1}, {A2}, {FGMW}, {G1}, {G2}, {N}, {NS}, {Tan}}.
Our first result is the following:

\renewcommand{\thedummy}{\ref{thm:main1}}
\begin{dummy}
Let $K$ be a ribbon  knot admitting an annulus presentation and
 $K_n$ $(n \in \Z)$ the knot obtained from $K$ by the $n$-fold annulus twist.
 Then   the homotopy 4-ball  $W(K_n)$  associated to  $K_n$ is diffeomorphic to $B^4$, that is,
\[   W(K_n) \approx  B^4.\]
In particular, 
$K_n$ is a slice knot.
\end{dummy}

Here recall the slice-ribbon conjecture.
A knot in  $S^3=\partial B^4$ is called \textit{slice}
if it bounds a smoothly embedded disk in $B^4$.
A knot in $S^3$ is called \textit{ribbon} 
if it bounds a smoothly immersed disk in $S^3$ 
with only ribbon singularities.
It is well known that every ribbon knot is  slice.
The \textit{slice-ribbon conjecture} states that 
any slice knot is ribbon.
There are some affirmative results on the slice-ribbon conjecture,
see \cite{{CD}, {GJ}, {Le}, {Li}}.
On the other hand,
Gompf, Scharlemann and Thompson \cite{GST}
demonstrated slice knots which might not be ribbon.
Similarly,  
there is no apparent reason for the slice knots $K_n$ in Theorem \ref{thm:main1} 
to  be ribbon. 

Let $\mathcal{K}_n$ $(n \ge 0)$ be the knot 
 obtained from $8_{20}$ 
%in  Rolfsen's table of knots 
 (with an appropriate annulus presentation)
by the $n$-fold annulus twist.
These are the simplest slice knots obtained by our method,
and were  studied by Omae \cite{Om}  in  a different viewpoint.
We will prove that these slice knots are ribbon.
To prove this, 
we introduce  a sufficient condition for given slice knots to be ribbon.

\renewcommand{\thedummmmy}{\ref{lem:ribbon}}
\begin{dummmmy}
Let HD be a handle diagram of $B^4$.
Suppose that HD is changed into the empty  handle diagram of $B^4$
by handle slides, adding or canceling 1/2-handle pairs, and isotopies.
Then the belt sphere of  any 2-handle of $HD$ is a ribbon knot.
\end{dummmmy}

Our second result is the following.

\renewcommand{\thedummy}{\ref{thm:ribbon}}
\begin{dummy}
The slice knot $\mathcal{K}_n$ $(n \ge 0)$ is ribbon.
\end{dummy}

We outline the proof as follows:
By the construction,
$\mathcal{K}_n$ $(n \ge 0)$ is  isotopic to the belt-sphere
of  a $2$-handle of  a certain handle diagram $HD$ of $B^4$  without 3-handles,
see the proof of Lemma  \ref{h_slice}.
By (rather long) handle calculus,
we prove that 
$HD$ is changed into the empty  handle diagram of $B^4$
by handle slides, canceling 1/2-handle pairs,  and isotopies.
By Lemma \ref{lem:ribbon},
$\mathcal{K}_n$ is ribbon.

In Section 6,  we propose two conjectures.
The first one is  the  following.

\renewcommand{\thedummmy}{\ref{conj:2}}
\begin{dummmy}
Let $HD$ be a handle diagram  of $B^4$  without 3-handles.
Then the belt-sphere of any $2$-handle of $HD$
is a ribbon knot.
\end{dummmy}

Note that, if Conjecture \ref{conj:2} is true, then 
slice knots in Theorem \ref{thm:main1}  and Gompf, Scharlemann and Thompson's slice knots
in \cite{GST} are  ribbon.
In this sense,  
to solve Conjecture \ref{conj:2} is the first step toward 
an affirmative answer   to the slice-ribbon conjecture.
For the details, see Section 6.

This paper is organized  as follows:
In Section 2, 
we recall some definitions which we will use.
In Section 3, 
we prove the main result (Theorem \ref{thm:main1}).
First, 
we give a handle decomposition of $W(K_n)$.
After adding a canceling 2/3-handle pair to  $W(K_n)$ suitably,
we prove  $W(K_n) \approx  B^4$.
In Section 4, 
we give an alternative proof of Theorem \ref{thm:main1}
in a special case
by a log transformation.
In Section 5, 
we give a sufficient condition for  given slice knots
to be ribbon  (Lemma \ref{lem:ribbon}).
As an application,
we prove Theorem \ref{thm:ribbon}.
In Section 6, 
we give  two conjectures.\\

\nbf{Notations.} 
We denote by
 $M_K(n)$ the 3-manifold obtained
from $S^3$ by $n$-surgery on $K$ and
by $X_{K}(n)$ the smooth $4$-manifold obtained from 
$B^4$ by attaching a $2$-handle along $K$ with framing $n$.w
The symbol $\approx$ stands for a diffeomorphism.
We denote by $\mathcal{K}$ the knot $8_{20}$
and by  $\mathcal{K}_n$  $(n \in \Z)$ the knot obtained from $8_{20}$ with the annulus presentation
in Figure \ref{fig:Def-BP}.
In figures, we denote  by  $\sim$ an isotopy
and by  $\to$  a handle slide, a handle canceling or a blow-up.

\subsection*{Acknowledgments}
The first author was supported by JSPS KAKENHI Grant Numbers 23840021 and 13J05998.
The second  author was  supported by  JSPS KAKENHI Grant Number 24840006.
%The second  author was partially supported by KAKENHI,
%Grant-in-Aid for Research Activity start-up (No.~24840006),
%Japan Society for the Promotion of Science.
The first author thanks Charles Livingston for explaining Lemma \ref{AJOT} several years ago,
which is the starting point of this project.
He also thanks Riccardo Piergallini for  stimulating discussions on ribbon knots and ribbon disks.
We   thank the anonymous referees for very careful readings of this paper and  helpful comments.
\section{Preliminary}\label{sec:Knots} 

In this section, 
we define
an annulus twist, annulus presentation and
recall the knots constructed by Omae and homotopy 4-balls.\\
%\subsection{Annulus twist} \label{subsec:AnnulusTwist} 

\nbf{Annulus twist.} 
Let $V$ be the solid torus standardly embedded in $S^3$
and $V'$  the $3$-manifold as in Figure  \ref{fig:solid_torus}.
Then the following is known.

\begin{figure}[htb]
\includegraphics[width=1\textwidth]{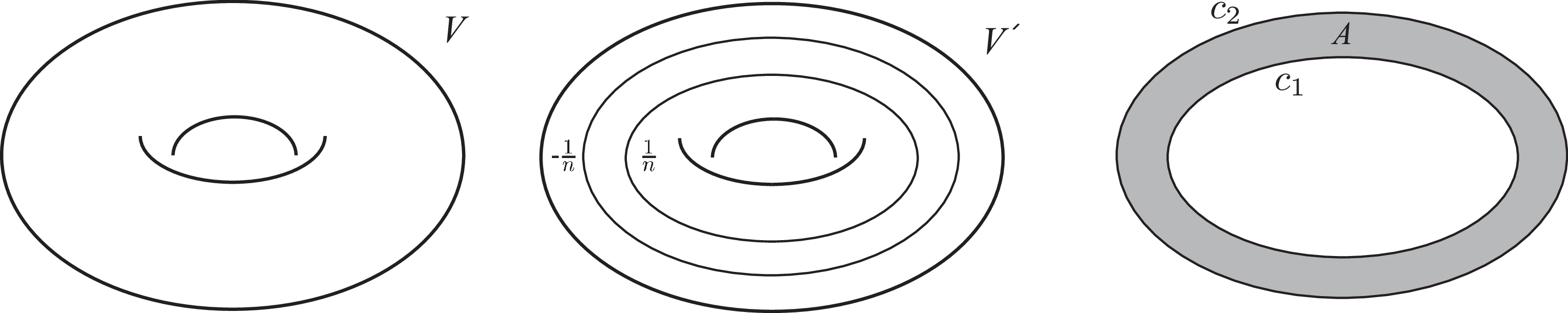}
\caption{The definitions of $V$,  $V'$, and $A$, $c_1$, $c_2$.}
\label{fig:solid_torus}
\end{figure}

\begin{lem} [cf.~Theorem 2.1 in \cite{Os}]\label{lem:Osoinach}
There exists a (natural) diffeomorphism 
\[ \varphi_n  : V' \longrightarrow  V\]
such that $\varphi_n|_{\partial V'} =id$.
\end{lem}

\begin{rmk}
Osoinach  \cite{Os} considered the diffeomorphism  $\varphi_n^{-1}$.
\end{rmk}

Let $A \subset \R^2 \cup \{ \infty \} \subset S^3$ be an embedded annulus 
and
set $\partial A = c_1 \cup c_2$ as in Figure \ref{fig:solid_torus}. 
An \textit{$n$-fold annulus twist along $A$} is the following operation:\\
\vskip -3mm

\noindent
(1) Regard $c_1$ 
as a $\dfrac{1}{n}$-framed knot
and $c_2$ as  a $-\dfrac{1}{n}$-framed knot for $n \in \Z$, and
%\footnote{That is, perform  $1/n$-surgery on $c_1$ and $-1/n$-surgery on $c_2$  for $n \in \Z$.}.\\
(2) take a solid torus $V'$ which is a  neighborhood of $A$, and\\
(3) apply the diffeomorphism $\varphi_n$ in Lemma \ref{lem:Osoinach}.\\ 
\vskip -3mm

\noindent
A $1$-fold annulus twist along $A$ is  called an
 \textit{annulus twist along $A$}.\\

%\subsection{Annulus presentation}\label{subsec:Band} 

\nbf{Annulus presentation.}
The first author, Jong, Omae and Takeuchi \cite{AJOT} introduced 
the notion of  an annulus presentation\footnote{In \cite{AJOT},
it was called a band presentation.}
of a knot for which we can associate an annulus.

We recall the definitions of 
an annulus presentation of a knot
  as follows:
Let $A \subset \R^2 \cup \{ \infty \} \subset S^3$ be a trivially embedded annulus
with an $\varepsilon$-framed unknot $c$ in  $S^3$ as shown in the left side of Figure~\ref{fig:Def-BP},
where $\varepsilon = \pm 1$. 
Take an embedding of a band $b$: $I \times I \to S^3$ such that 
\begin{itemize}
\item $b(I \times I) \cap \partial A = b(\partial I \times I)$, 
\item $b(I \times I) \cap \text{int} A$ consists of ribbon singularities, and
\item $b(I \times I)  \cap c= \emptyset$,
\end{itemize}
where $I = [0,1]$. 
Throughout this paper,
we assume that  $A \cup b(I \times I)$ is orientable.
This means that we deal with only 0-framed knots,
see \cite{AJOT}. 
For simplicity,
we also assume  that  $\varepsilon=-1$.
If a knot $K \subset  S^3$ is isotopic to the knot $\left( \partial A \setminus b(\partial I \times I)\right) \cup b( I \times \partial I)$
in $M_{c}(-1) \approx S^3$, 
then we say that $K$ admits an \textit{annulus presentation} $(A,b,c)$.
A typical example of an annulus presentation of a knot is given in Figure~\ref{fig:Def-BP}.

\begin{figure}[htb]
\includegraphics[width=1.0\textwidth]{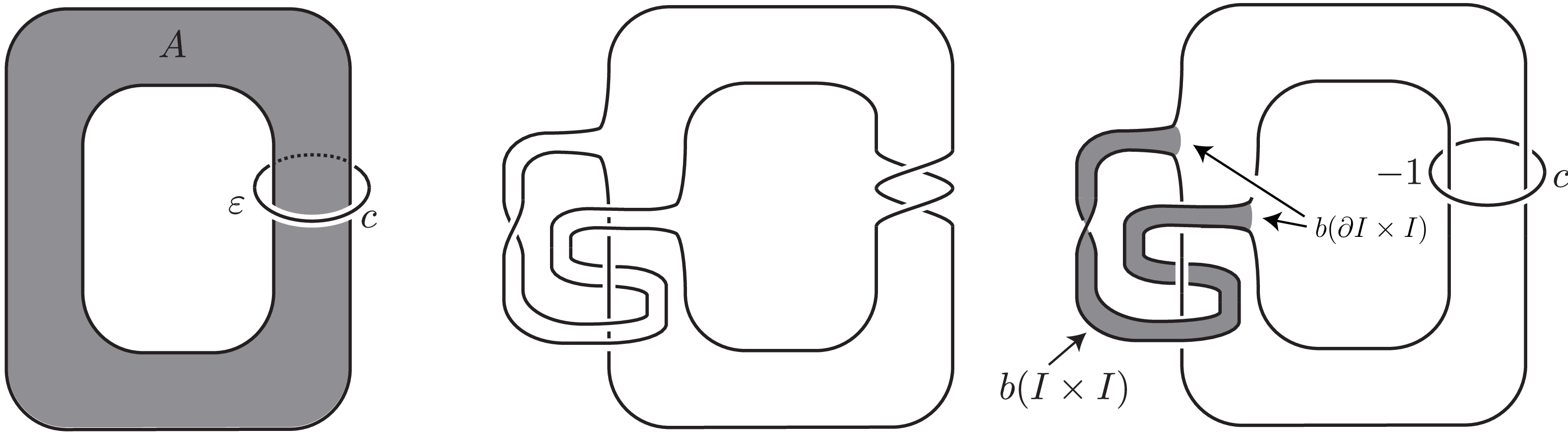}
\caption{The knot $8_{20}$ depicted in the center admits an annulus presentation as in the right side.}
\label{fig:Def-BP}
\end{figure}

\begin{figure}[htb]
\includegraphics[width=1.0\textwidth]{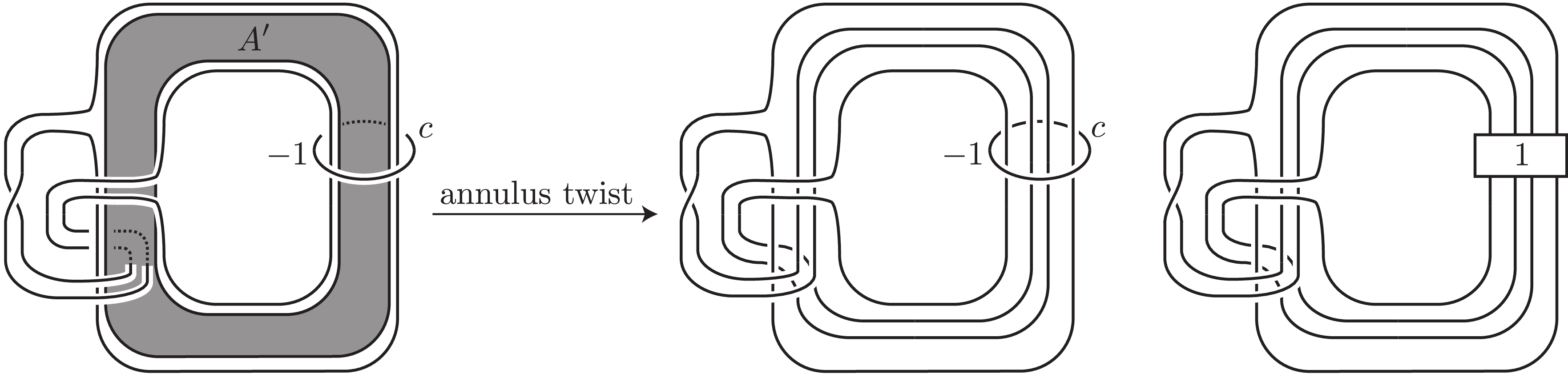}
\caption{The associated annulus $A'$ (left side), an annulus twist along $A'$, and the resulting knot (right side).}
\label{fig:Ex-AT}
\end{figure} 

Let $K$ be a knot admitting an annulus presentation $(A,b,c)$. 
Shrinking the annulus $A$ slightly,
we obtain an annulus $A' \subset A$ 
as shown in Figure~\ref{fig:Ex-AT}. 
We apply the $n$-fold  ($n \in \Z$) annulus twist  along $A'$
and  blow down the $-1$-framed unknot $c$.
Figure~\ref{fig:Ex-AT} illustrates the case $n=1$.
We call the resulting knot 
\textit{the knot obtained from $K$
by the $n$-fold annulus twist} without mentioning $A'$.
The first author, Jong, Omae and Takeuchi proved the following:

\begin{lem} [\cite{AJOT}] \label{AJOT}
Let $K$ be a knot admitting an annulus presentation
and $K_n$ ($n \in \Z$)  the knot obtained
from $K$ by the $n$-fold annulus twist.
Then 
\[ 
M_{K}(0)  \approx M_{K_n}(0).
\]
If $K$ is a slice knot, 
then $K_n$ bounds a smoothly embedded disk in a homotopy 4-ball $W(K_n)$
such that $\partial W(K_n) \approx S^3$.
\end{lem}

\begin{rmk}
Under the assumption of Lemma \ref{AJOT},
we can also prove that $X_{K}(0) \approx  X_{K_n}(0)$, 
see  \cite{AJOT}. 
The homotopy 4-ball $W(K_n)$ in Lemma \ref{AJOT} depends on the choice of 
a diffeomorphism between $M_{K}(0)$ and $M_{K_n}(0)$.
\end{rmk}

\nbf{The knots obtained from $8_{20}$ and  homotopy 4-balls.} 
Let $8_{20}$ be the knot in the center of Figure \ref{fig:Def-BP}.
Then it admits an annulus presentation, 
see the right side of Figure \ref{fig:Def-BP}.
Let  $\mathcal{K}_n$ be the knot obtained
from $8_{20}$ by the $n$-fold annulus twist.
Omae studies these knots in \cite{Om}.
We prove the following two lemmas.

\begin{lem} \label{h_slice} 
The above knot $\mathcal{K}_n$ $(n \ge 0 )$ bounds a smoothly embedded disk 
in  a homotopy 4-ball $W_n$
such that $\partial W_n \approx S^3$ and 
it has the handle decomposition as in Figure \ref{Wn1}.
\end{lem}

%%%%%%%%%%%%%%%%5
\begin{figure}
\includegraphics[width=0.5\textwidth]{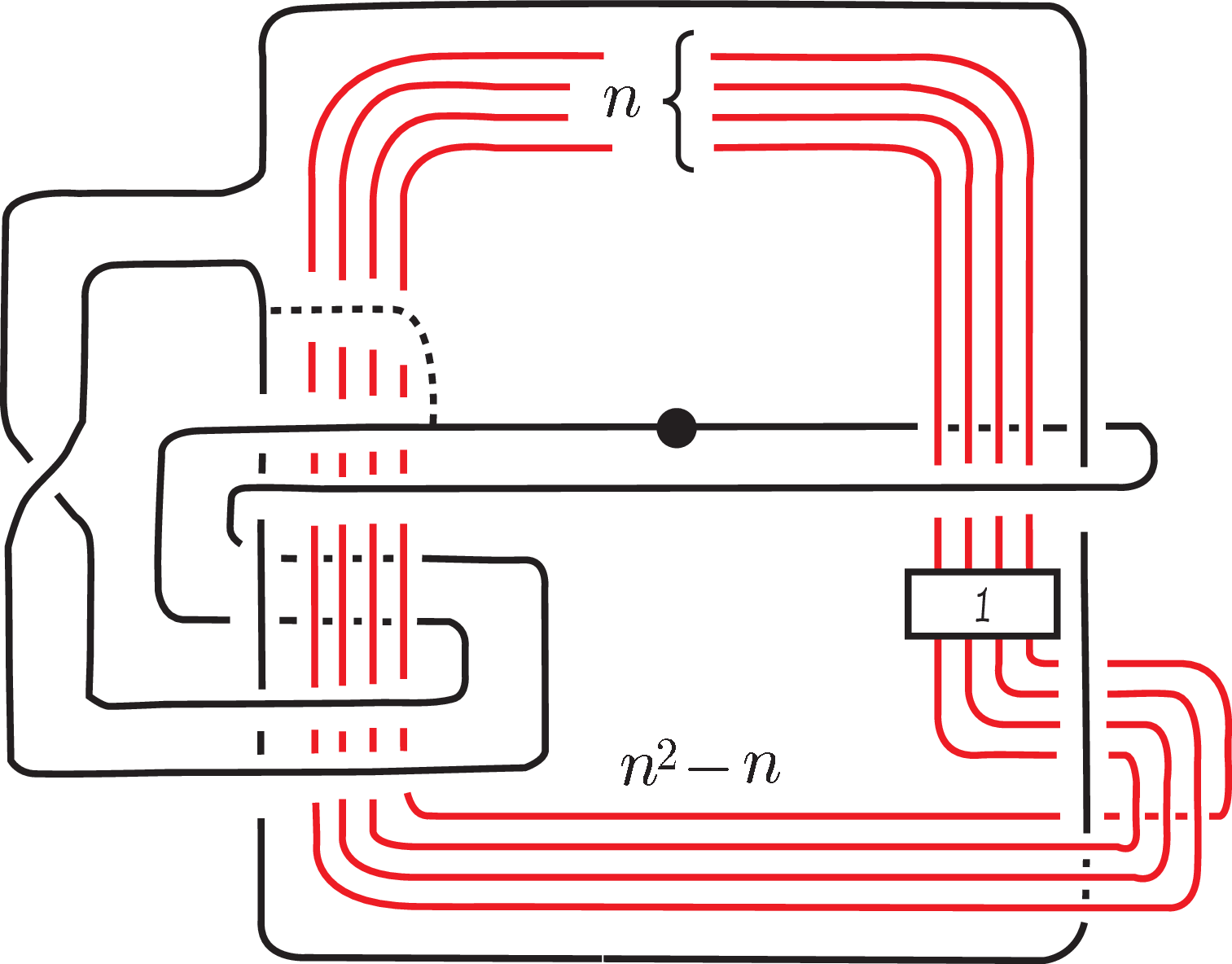}
\caption{A handle decomposition of $W_n$ $(n \ge 0 )$.}
\label{Wn1}
\end{figure} 
%%%%%%%%%%%%%%%%%%
For the dotted circle notation for the complements of ribbon disks,
see subsection $1.4$ in  \cite{A} (see also subsection $6.2$ in \cite{GS}).
The first half of this lemma follows from
Lemma \ref{AJOT}.
For the sake of completeness,
we give the proof.

\begin{proof}
Let $f_n : M_{\mathcal{K}_0}(0)  \to M_{\mathcal{K}_n}(0)$ be the diffeomorphism
described in   Figure \ref{diffeo} (here we ignore the framed knots colored red).
For the detail of this  diffeomorphism, see \cite{Te}.

\begin{figure}[!hbt]
\includegraphics[width=1.0\textwidth]{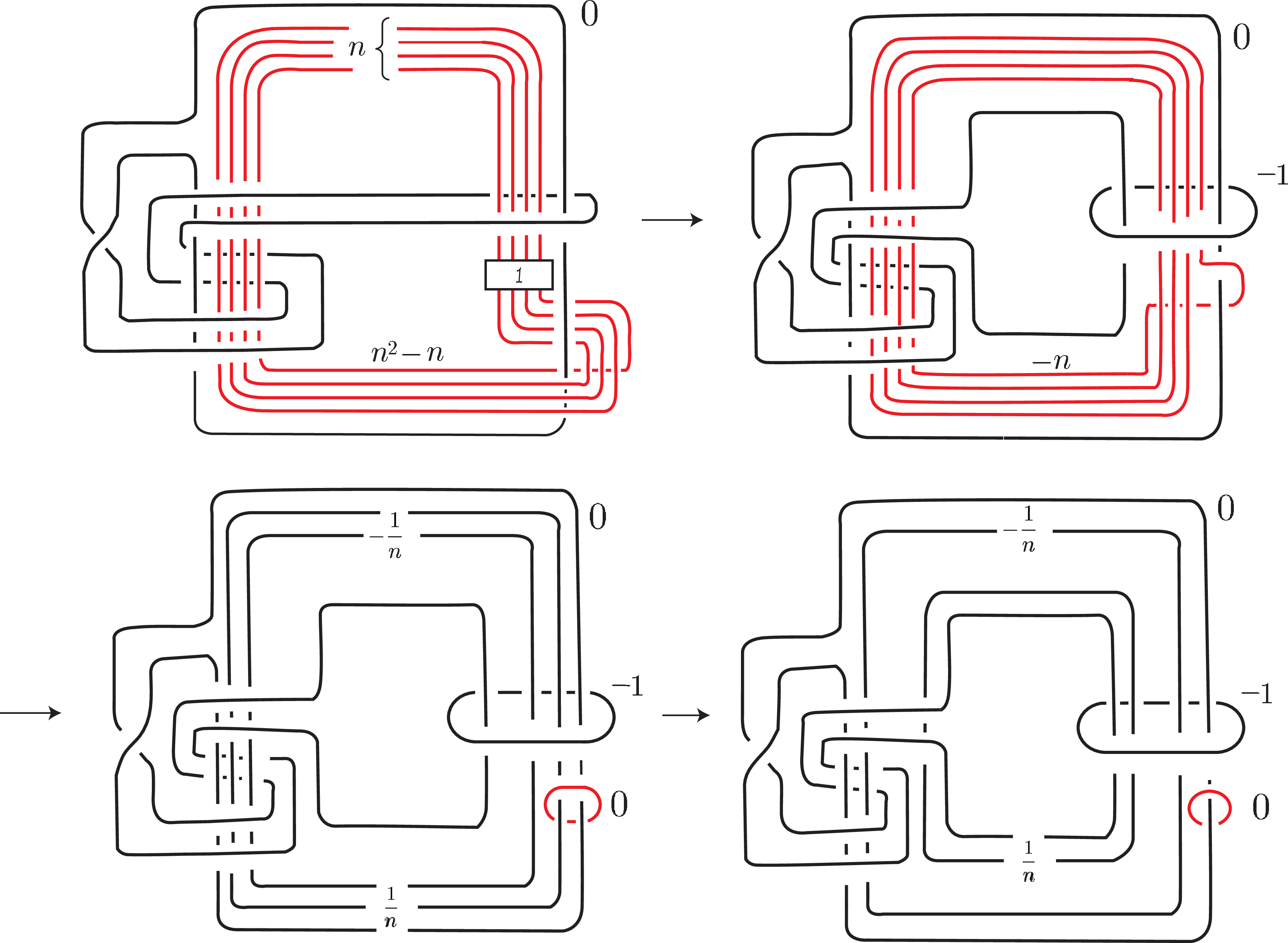}
\caption{A diffeomorphism from $M_{\mathcal{K}_0}(0)$ to  $M_{\mathcal{K}_n}(0)$.
$M_{\mathcal{K}_0}(0)$ is represented by the first picture.
The second picture is obtained by a blow up.
The third picture is obtained by applying $\varphi_n^{-1}$ in Lemma \ref{lem:Osoinach}.
The last picture is obtained by a handle slide.
Then we obtain  $M_{\mathcal{K}_n}(0)$ from the last picture
by applying $\varphi_n$ in Lemma \ref{lem:Osoinach} and a blow down.}
\label{diffeo}
\end{figure}

The knot $\mathcal{K}_0$ is ribbon.
Indeed, if we add a band along the dashed arc 
as in the left side of Figure \ref{ribbon},
then we obtain the two component unlink.
\begin{figure}
\includegraphics[width=0.8\textwidth]{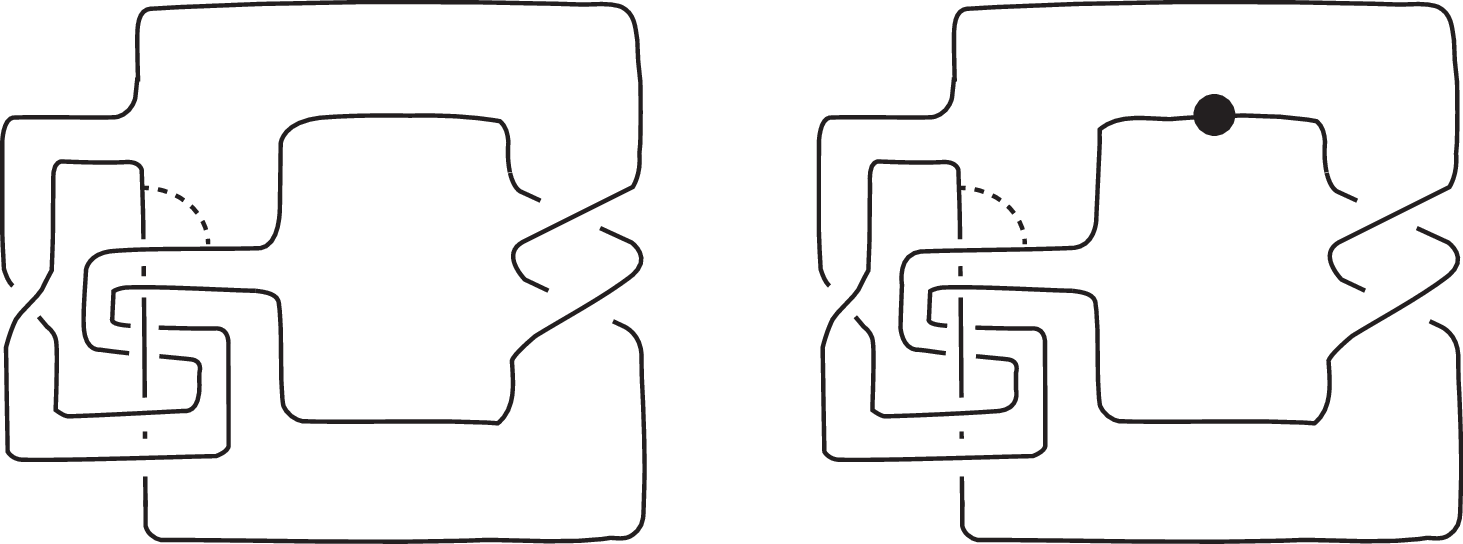}
\caption{$\mathcal{K}$ with a dashed arc  and the handle decomposition of $X$.}
\label{ribbon}
\end{figure} 
\begin{figure}
\includegraphics[width=1.0\textwidth]{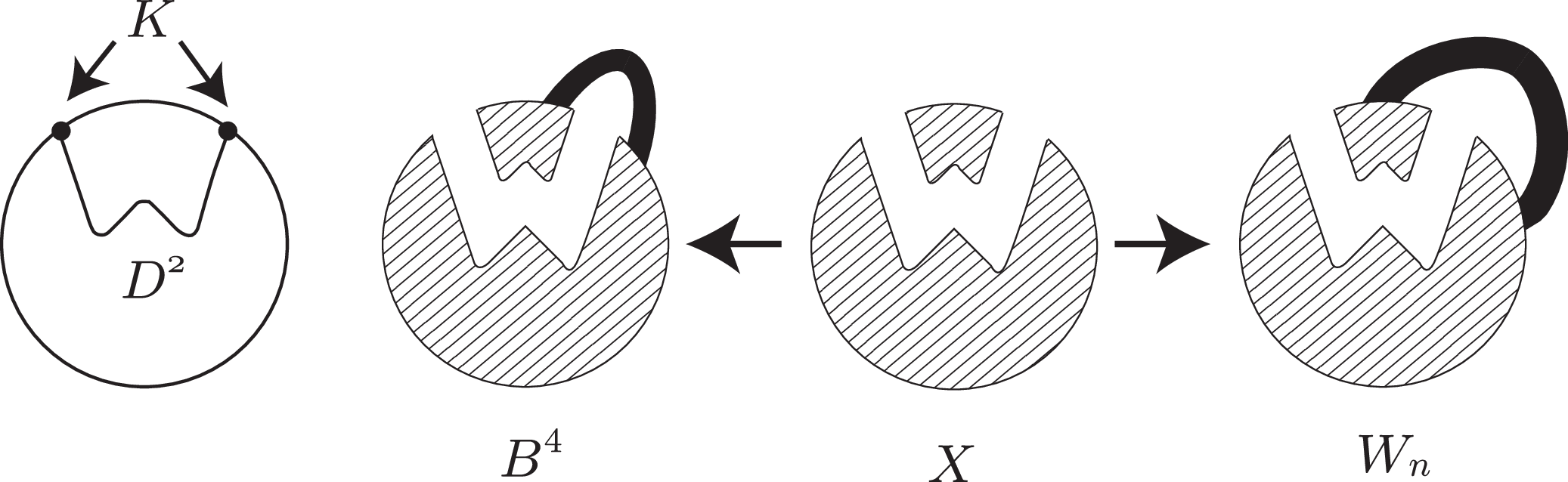}
\caption{}
\label{fig:h_4-ball}
\end{figure}
%%%%%%%%%%%%%%%%
Let $D^2$ be the corresponding smoothly, properly embedded disk in $B^4$
such that $\partial D^2=\mathcal{K}_{0}$
and  $X$  the 4-manifold obtained from $B^4$ 
by removing  an open tubular neighborhood of $D^2$ in $B^4$
(see Figure \ref{fig:h_4-ball}).
Note that  $\partial X$ is (naturally) diffeomorphic to $M_{\mathcal{K}_0}(0)$.
If we attach a 2-handle along the meridian of $\mathcal{K}_0$
in $M_{\mathcal{K}_0}(0) \approx \partial X$ with framing $0$, 
then the resulting $4$-manifold is diffeomorphic to $B^4$.
The homotopy 4-ball $W_n$ is obtained from $X$
by attaching a 2-handle along the meridian $\mu_n$ of $\mathcal{K}_n$
in $M_{\mathcal{K}_n}(0) \approx \partial X$ with framing $0$.
Schematic pictures are given in Figure \ref{fig:h_4-ball}.
The knot $\mathcal{K}_n$ is isotopic to the boundary of the cocore disk of the
$2$-handle attached along  $\mu_n$.
Thus $\mathcal{K}_n$ bounds the cocore disk in $W_n$, that is, a smoothly
embedded disk in $W_n$.

Next, we draw a handlebody picture of $W_n$.
Recall that  $X$ has the handle decomposition  
as in the right of Figure \ref{ribbon}.
The diffeomorphism from $\partial X$ to $M_{\mathcal{K}_0}(0)$, 
denoted by $g$,
is given by changing the dot to $0$.
By the construction,
 $W_n$ is obtained from $X$
by attaching a 2-handle along $(f_{n} \circ g)^{-1} (\mu_n)$ in  $\partial X$ with a suitable framing.
By Figure \ref{diffeo}, 
the framing is $n^2-n$ and 
$W_n$ has the handle decomposition  as in Figure \ref{Wn1}.
%It is not difficult to check that Omae's knot $K_n$ is located as in Figure \ref{Wn1}.
\end{proof}

\begin{lem} \label{h_slice2} 
The above knot $\mathcal{K}_n$ $(n < 0 )$ bounds a smoothly embedded disk in  a homotopy 4-ball $W_n$
such that $\partial W_n \approx S^3$ and it has the handle decomposition as in Figure \ref{W_{-m}1}.
\end{lem}

%%%%%%%%%%%%%%%%5
\begin{figure}
\centering
\begin{overpic}[width=7cm, clip]{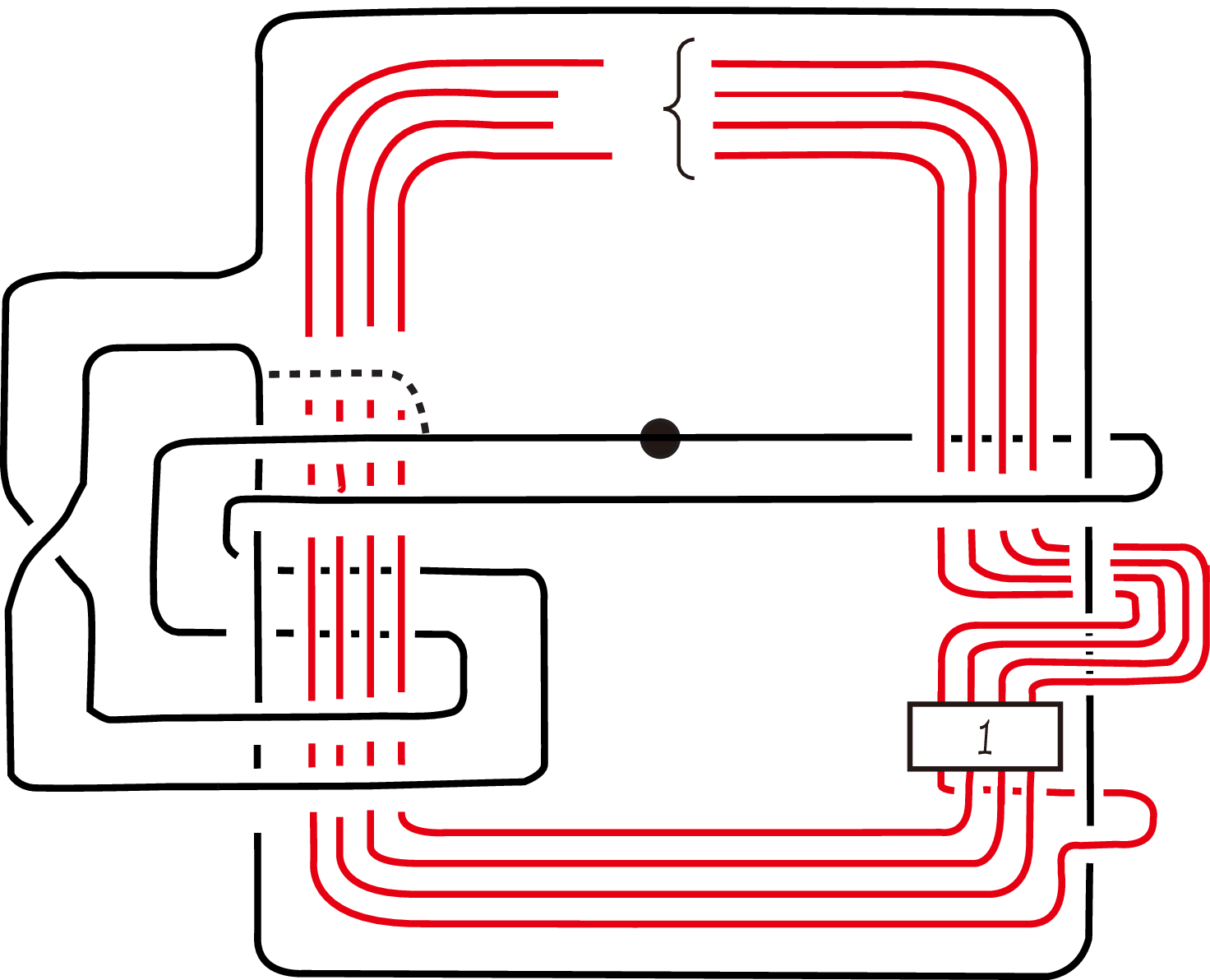}
\label{W_{-m}1}
  \put(87,140){$|n|$}
  \put(90,27){$n^2-n$}
\end{overpic}
\caption{A handle decomposition of $W_{n}$ $(n < 0 )$.}
\label{W_{-m}1}
\end{figure} 
%%%%%%%%%%%%%%%%%%%

\begin{proof}
Set $n=-m$ for some positive integer $m$.
Let $f_{-m} : M_{\mathcal{K}_0}(0)  \to M_{\mathcal{K}_{-m}}(0)$ be the diffeomorphism
described in  Figure \ref{diffeo2} (here we ignore the framed knots colored red).

\begin{figure}[!hbt]
\centering
\begin{overpic}[width=15cm,clip]{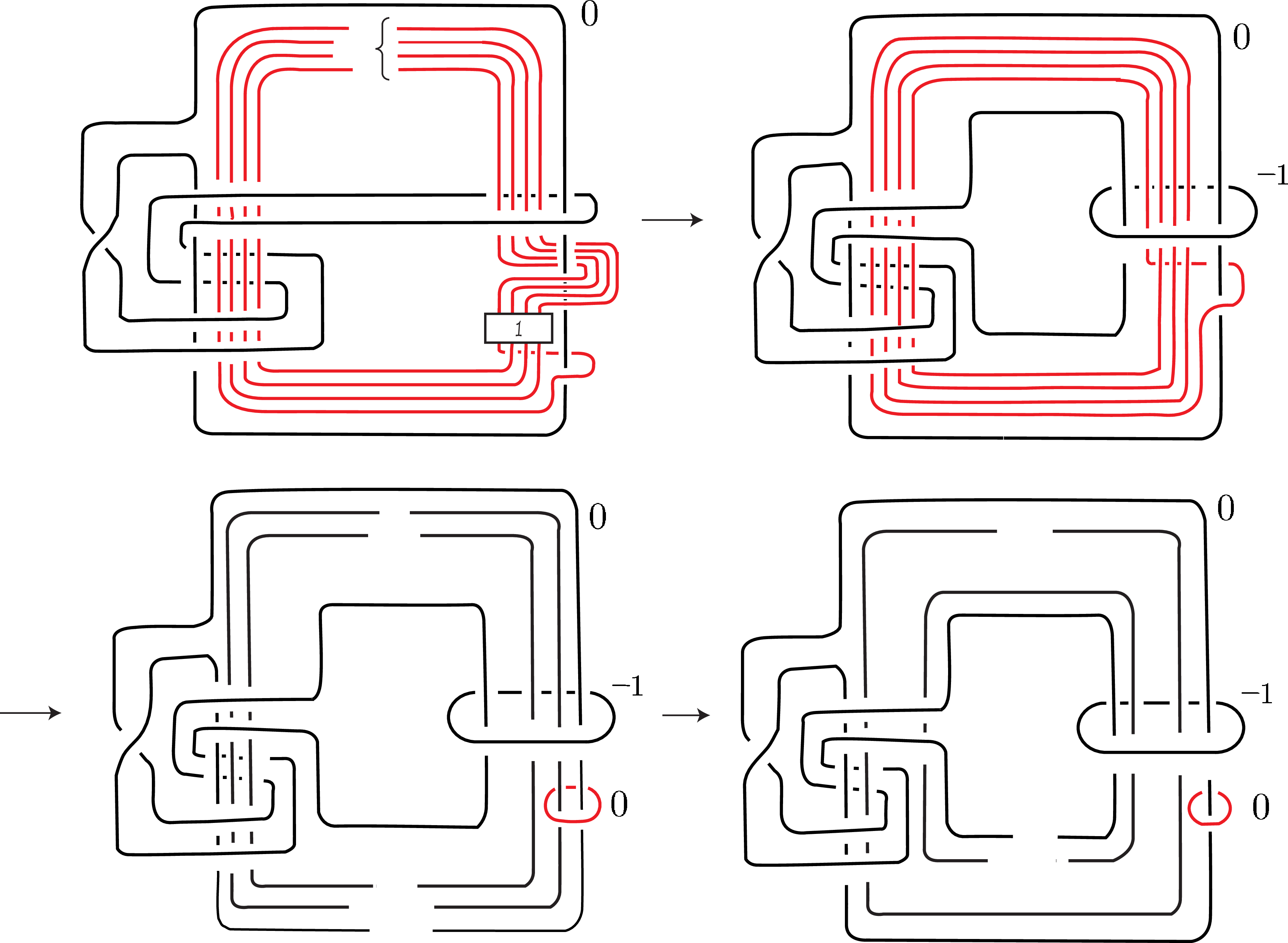}
  \put(113,293){$m$}
  \put(114,192){$m^2+m$}
  \put(340,190){$m$}
  \put(126,132){${\frac{1}{m}}$}
  \put(118,9){${-\frac{1}{m}}$}
  \put(335,133){${\frac{1}{m}}$}
  \put(332,25){$-{\frac{1}{m}}$}
\end{overpic}
\caption{A diffeomorphism from $M_{\mathcal{K}_0}(0)$ to  $M_{\mathcal{K}_{-m}}(0)$.
$M_{\mathcal{K}_0}(0)$ is represented by the first picture.
The second picture is obtained by a blow up.
The third picture is obtained by applying $\varphi_{-m}^{-1}$ in Lemma \ref{lem:Osoinach}.
The last picture is obtained by a handle slide.
Then we obtain  $M_{\mathcal{K}_{-m}}(0)$ from the last picture
by applying $\varphi_{-m}$ in Lemma \ref{lem:Osoinach} and a blow down.}
\label{diffeo2}
\end{figure}

The knot $\mathcal{K}_0$ is ribbon.
Indeed, if we add a band along the dashed arc 
as in the left side of Figure \ref{ribbon},
then we obtain the two component unlink.
Let $D^2$ be the corresponding smoothly, properly embedded disk in $B^4$
such that $\partial D^2=\mathcal{K}_{0}$
and  $X$  the 4-manifold obtained from $B^4$ 
by removing  an open tubular neighborhood of $D^2$ in $B^4$.
Note that  $\partial X$ is (naturally) diffeomorphic to $M_{\mathcal{K}_0}(0)$.
The homotopy 4-ball $W_{-m}$ is obtained from $X$
by attaching a 2-handle along the meridian $\mu_{-m}$ of $\mathcal{K}_{-m}$
in $M_{\mathcal{K}_{-m}}(0) \approx \partial X$ with framing $0$.
The knot $\mathcal{K}_{-m}$ is isotopic to the boundary of the cocore disk of the
$2$-handle attached along  $\mu_{-m}$.
Thus $\mathcal{K}_{-m}$ bounds the cocore disk in $W_{-m}$, that is, a smoothly
embedded disk in $W_{-m}$.

Next, we draw a handlebody picture of $W_{-m}$.
Recall that  $X$ has the handle decomposition  
as in the right of Figure \ref{ribbon}.
The diffeomorphism from $\partial X$ to $M_{\mathcal{K}_0}(0)$, 
denoted by $g$,
is given by changing the dot to $0$.
By the construction,
 $W_{-m}$ is obtained from $X$
by attaching a 2-handle along $(f_{{-m}} \circ g)^{-1} (\mu_{-m})$ in  $\partial X$ with a suitable framing.
By Figure \ref{diffeo2}, 
the framing is $m^2+m(=n^2-n)$.
Therefore  
$W_{-m}(=W_n)$ has the handle decomposition  as in Figure \ref{W_{-m}1}.
%It is not difficult to check that Omae's knot $K_n$ is located as in Figure \ref{Wn1}.
\end{proof}

\section{A construction of slice knots via annulus twists.}
In this section, 
we prove the following theorem  by introducing a canceling 2/3-handle pair.

\begin{thm} \label{thm:main1}
Let $K$ be a ribbon  knot admitting an annulus presentation and
 $K_n$ $(n \in \Z)$ the knot obtained from $K$ by the $n$-fold annulus twist.
 Then   the homotopy 4-ball  $W(K_n)$  associated to  $K_n$ is diffeomorphic to $B^4$, that is,
\[   W(K_n) \approx  B^4.\]
In particular, 
$K_n$ is a slice knot.
\end{thm}

\begin{proof}
First we consider the case $\mathcal{K}=8_{20}$ with 
the annulus presentation 
as the right side of Figure \ref{fig:Def-BP} and $n \ge 0$.
By Lemma \ref{h_slice},
$\mathcal{K}_n$ bounds a smoothly embedded disk in the homotopy 4-ball $W_n$
given by the picture in Figure \ref{Wn1}.
We prove the  following claims.\\

\nbf{Claim 1.} $W_n$ $(n \ge 0)$ also has the handle decomposition given by the picture in Figure~\ref{Wn2}.

\begin{figure}[htpb]
\begin{center}
\includegraphics[width=.5\textwidth]{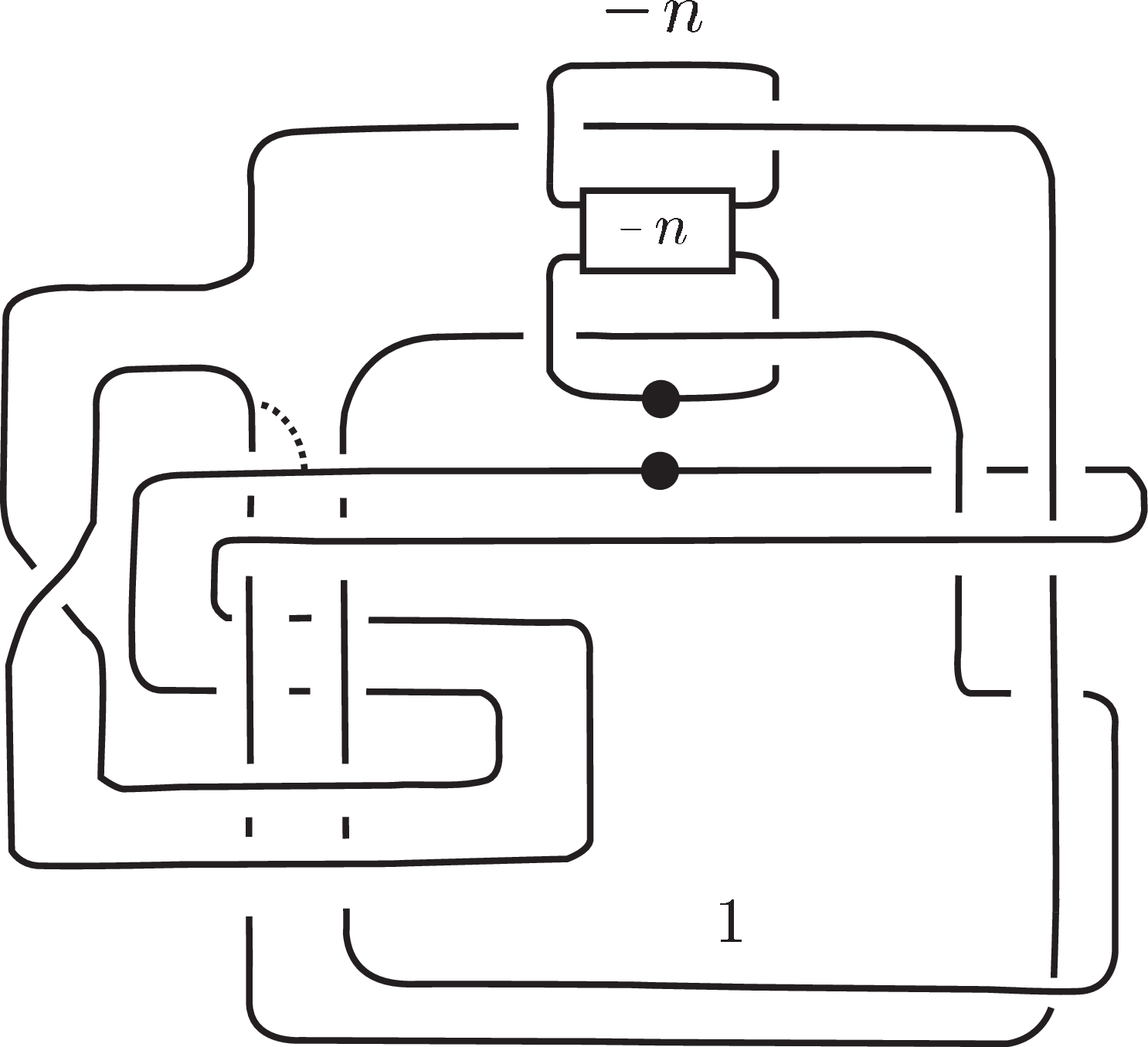}
\caption{A handle decomposition of $W_n$.}
\label{Wn2}
\end{center}
\end{figure}

\begin{proof}
Inserting a canceling 1/2-handle pair to $W_n$, 
we obtain the first picture in Figure~\ref{deformation1}.
Note that, in Figure ~\ref{deformation1},
we ignore the dashed arc because it is disjoint from the handle slides below.
By handle slides, 
we obtain the second picture.
By inserting a canceling 1/2-handle pair to $W_n$
and handle slides,
we obtain the third picture.
After a 1-handle slide (and a 2-handle slide, annihilating a canceling 1/2-handle pair
and isotopy), we obtain the  last picture.
Therefore, 
$W_n$ has the handle decomposition given by the picture in Figure~\ref{Wn2}.
\end{proof}

%%%%%%%%%%%%%%%%%%%%%%%%%%%%%%%
\begin{figure}[htpb]
\begin{center}
\includegraphics[width=1.\textwidth]{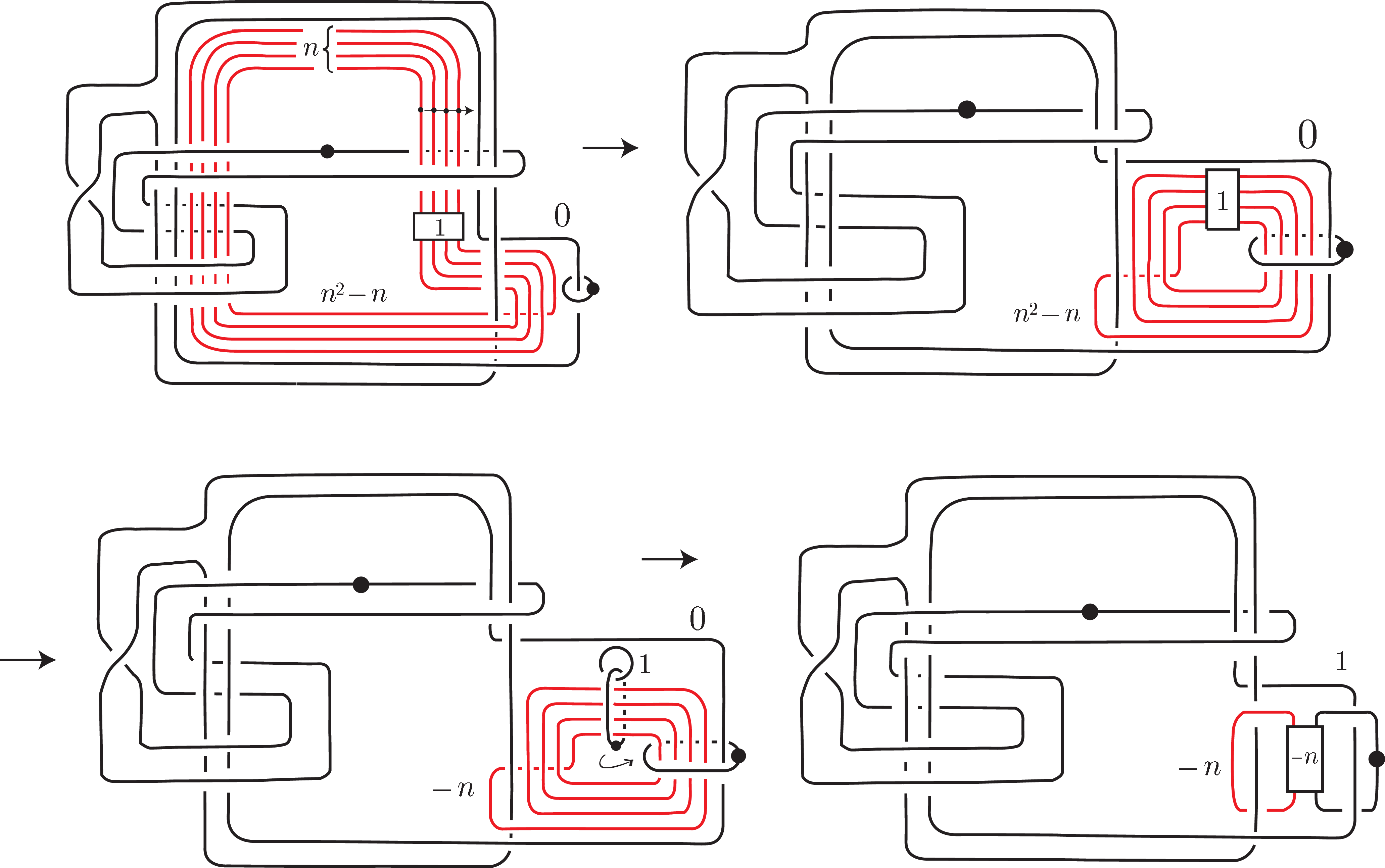}
\caption{Handle decompositions of  $W_n$ $(n \ge 0)$.}
\label{deformation1}
\end{center}
\end{figure}
%%%%%%%%%%%%%%%%%%%%%%%%%%%%%%%%%%

\nbf{Claim 2.} $W_n \approx W_{n-1}$. 

\begin{proof}
We show that $\gamma, \lambda \subset \partial  W_n$ described in  Figure~ \ref{deformation2} are isotopic and 
each curve is the unknot in $\partial  W_n= S^3$.
By Claim~$1$, 
$W_n$ has the handle decomposition 
given by the first picture in Figure~\ref{deformation2}.
We replace the two dotted circles with the
zero-framed circles.
Then we obtain the second picture in Figure \ref{deformation2}.
Handle calculus in Figure \ref{deformation2}  illustrates 
the  diffeomorphism from $\partial  W_n$ to $S^3$.

%%%%%%%%%%%%%%%%%%%%%%%%%%%%%
\begin{figure}[htpb]
\begin{center}
\includegraphics[width=1.0\textwidth]{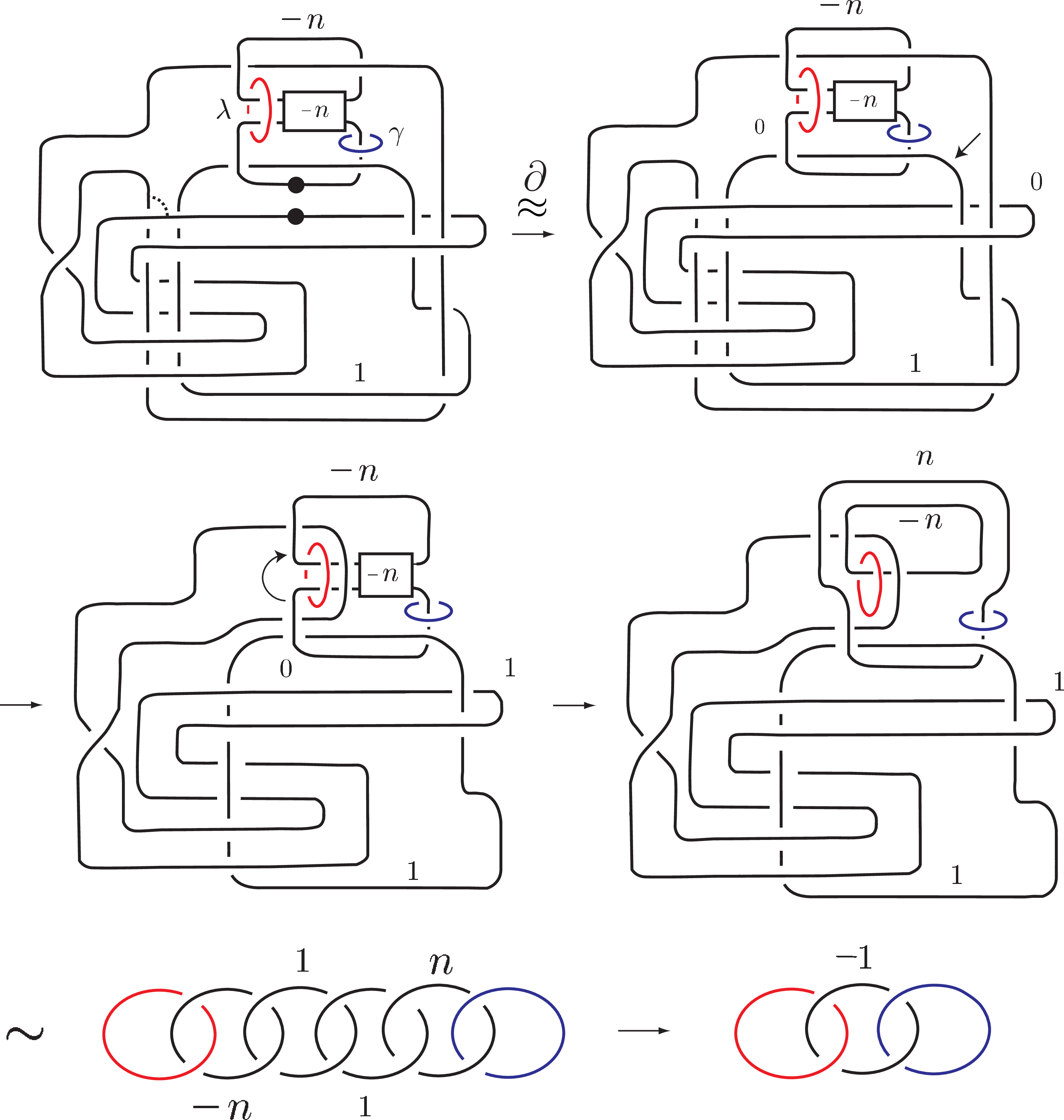}
\caption{A specific diffeomorphism identifying $\partial  W_n$ with $S^3$
which tells us that two curves  $\gamma, \lambda \subset \partial  W_n$ are isotopic.}
\label{deformation2}
\end{center}
\end{figure}
%%%%%%%%%%%%%%%%%%%%%%%%%%

Furthermore,
if we regard $\gamma$ (or $\lambda$) as a $-1$-framed knot,
then it is isotopic to the $0$-framed unknot in  $S^3$.
Now we insert  a canceling  2/3-handle pair to  $W_n$.
Then $W_n$ is diffeomorphic to 
the first picture in  Figure~\ref{deformation3}.
By a handle slide,
we obtain the second picture,
which is diffeomorphic to  $W_{n-1}$.
%%%%%%%%%%%%%%%%%%%%%%%%%%%%%
\begin{figure}[htpb]
\begin{center}
\includegraphics[width=1.0\textwidth]{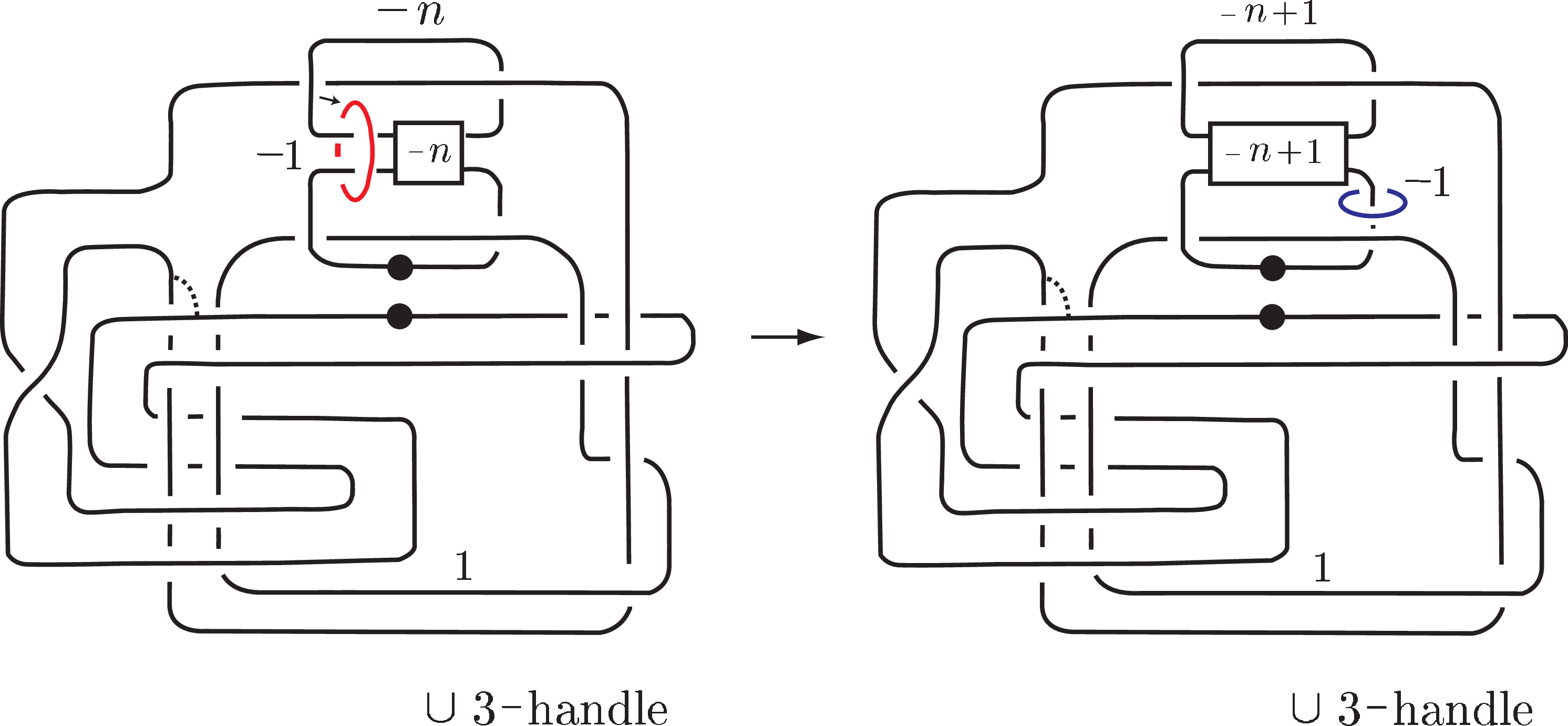}
\caption{A handle slide.}
\label{deformation3}
\end{center}
\end{figure}
%%%%%%%%%%%%%%%%%%%%%%%%%%%%
\end{proof}

\noindent
By Claim $2$, 
$W_n \approx W_{n-1}  \approx \cdots \approx W_{1} \approx W_{0}$.
By the construction,
$W_{0} \approx  B^4$.
Therefore $W_{n} \approx  B^4$
and $\mathcal{K}_n$ is a slice knot.

Next we consider the case $\mathcal{K}=8_{20}$ with 
the annulus presentation 
as the right side of Figure \ref{fig:Def-BP} and $n < 0$.
By Lemma \ref{h_slice2},
$\mathcal{K}_n$ bounds a smoothly embedded disk in the homotopy 4-ball $W_n$
given by the picture in Figure \ref{W_{-m}1}.
We prove the  following claim.\\

\nbf{Claim 3.} $W_n$ $(n<0)$ also has the handle decomposition given by the picture in Figure~\ref{Wn2}.

\begin{proof}
Inserting a canceling 1/2-handle pair to $W_n$, 
we obtain the first picture in Figure~\ref{deformation1_1}.
Note that, in Figure ~\ref{deformation1_1},
we ignore the dashed arc because it is disjoint from the handle slides below.
By a similar handle calculus to that in Figure \ref{deformation1},
 we obtain the  second picture.
Therefore, 
$W_n$ has the handle decomposition given by the picture in Figure~\ref{Wn2}.
\end{proof}

%%%%%%%%%%%%%%%%%%%%%%%%%%%%%%%
\begin{figure}[htpb]
\begin{overpic}[width=15cm,clip]{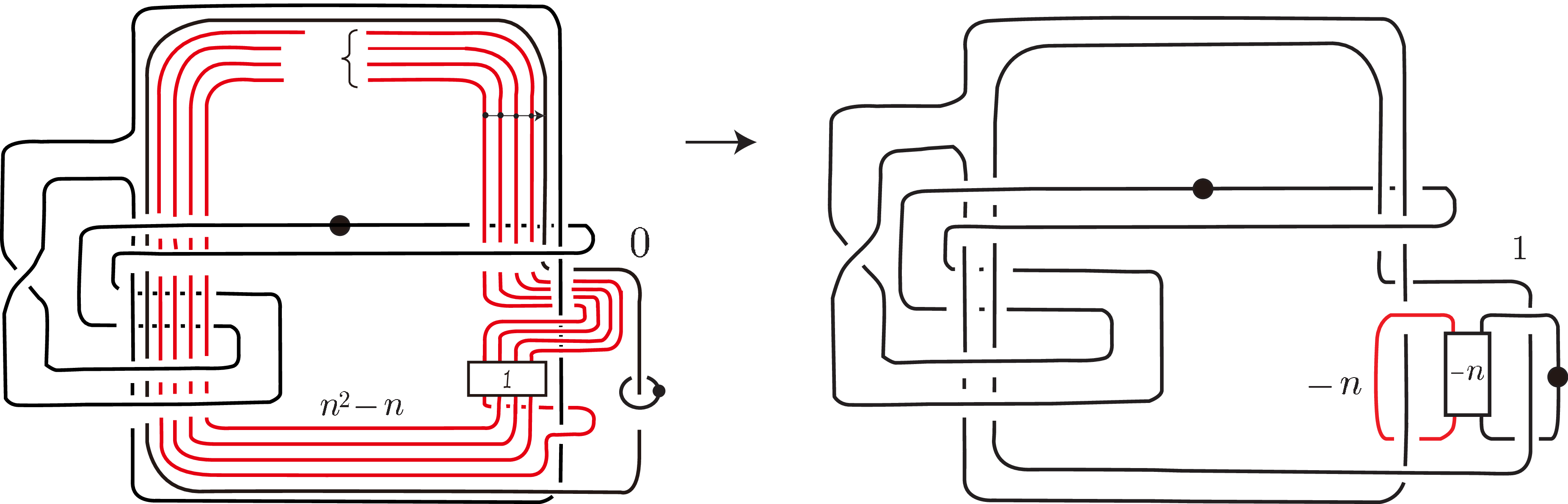}
  \put(80,118){$|n|$}
\end{overpic}
\caption{Handle decompositions of  $W_n$ $(n<0)$.}
\label{deformation1_1}
\end{figure}
%%%%%%%%%%%%%%%%%%%%%%%%%%%%%%%%%%

By the same argument as that in Claim 2,
we can prove that 
$W_{n} \approx  B^4$
and $\mathcal{K}_n$ is a slice knot.

Now we consider the general case.
First suppose that $n \ge 0$.
In this case, 
we can also associate a  diffeomorphism
$f_n : M_{K}(0)  \to M_{K_n}(0)$ 
as described  in  Figure \ref{diffeo}.
Let $\mu_n$ be the meridian of $K_n$ in $M_{K_n}(0)$.
Then $f_n^{-1}(\mu_n)$ is 
as in the first picture in Figure \ref{general}
 (after ignoring the framing).
Since $K$ is ribbon,
there exist mutually disjoint bands $B_1, \cdots, B_{m}$ 
such that 
if we surgery along these bands, 
then we obtain the $(m+1)$-component unlink.
Furthermore, (by deforming these bands slightly)
we can assume that $B_i \cap f_n^{-1}(\mu_n) = \emptyset$ for each $i \in \{1, 2, \cdots, m\}$.
Then, as the proof of  Lemma \ref{h_slice},
we see that $K_n$ bounds a smoothly embedded disk in a homotopy $4$-ball $W(K_n)$
which has the handle decomposition as in the second picture in Figure \ref{general}.
Note that
we do not draw dashed arcs in Figure \ref{general}.
It is proved that  $W(K_n)$ also has 
the handle decomposition as in the third  picture in Figure \ref{general} similarly.
Then we can prove that $W(K_n) \approx B^4$ by the same argument.
Therefore $K_n$ is a slice knot.

For the case $n  < 0$,
by a similar argument to that in Claim 3,
$K_n$ bounds a smoothly embedded disk in a homotopy $4$-ball $W(K_n)$
which has the handle decomposition as in the third picture in Figure \ref{general}.
Then we can prove that $W(K_n) \approx B^4$ by the same argument again.
Therefore $K_n$ is a slice knot.

%%%%%%%%%%%%%%%%%%%%%%%%%%%%%
\begin{figure}[htpb]
\begin{center}
\includegraphics[width=1.0\textwidth]{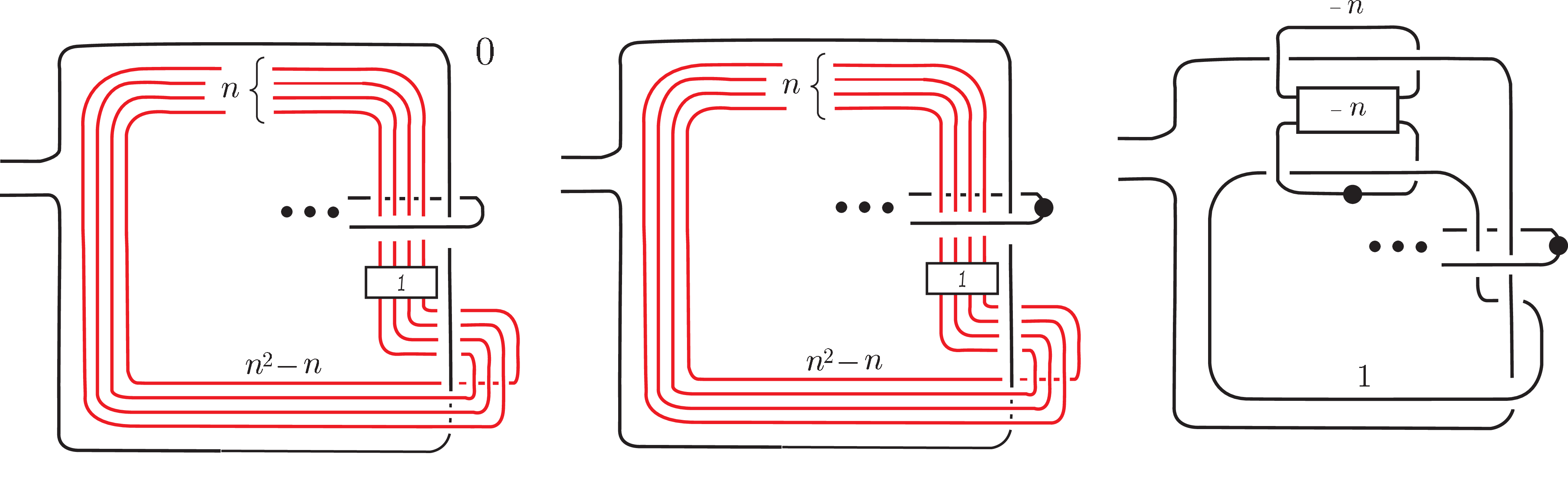}
\caption{}
\label{general}
\end{center}
\end{figure}
%%%%%%%%%%%%%%%%%%%%%%%%%%%%.

%%%%%%%%%%%%%%%%%%%%%%%%%%%%%%

\end{proof}

\section{Log transformation and fishtail neighborhood} \label{sec:log} 
In this section, 
we give an alternative proof of Theorem~\ref{thm:main1}
in the case $\mathcal{K}=8_{20}$.
More precisely, we prove that $W_n$ and $W_0$ are related by a log transformation
along a certain torus in $W_n$,
where $W_n$ is the homotopy 4-ball given by the picture in Figure \ref{Wn2}.
Lemma \ref{Gompf} due to Gompf
ensures that $W_n$ and $W_0$ are diffeomorphic,
which implies that  $W_n\approx B^4$.\\

\nbf{Log transformation.} 
Let $X$ be an oriented 4-manifold, 
$T$ an embedded torus with $T \cdot T=0$
and  $\varphi :  T^2 \times \partial D^2 \to \partial {\nu}(T)$ a diffeomorphism,
where $\nu (T) (\approx T^2 \times D^2)$ is a closed neighborhood of $T$ in $X$.
Removing int $\nu(T)$ from $X$ and attaching $T^2 \times D^2$ 
by $\varphi$, 
we obtain 
$$(X-\text{int}\ \nu(T))\cup_\varphi T^2 \times D^2.$$
Suppose that 
$$\varphi _{*} ([\{\text{pt.}\} \times \partial D^2])= p [\{\text{pt.}\} \times \partial D^2]+q[\gamma \times \{\text{pt.}\}]$$
for some essential simple closed curve  $\gamma$ in $T$.
Then we call this surgery a \textit{logarithmic transformation with multiplicity $p$, 
direction $\gamma$ and  auxiliary multiplicity $q$}.
If $p=1$, 
we call this logarithmic transformation a \textit{$q$-fold Dehn twist along $T$ parallel to $\gamma$}.\\

\nbf{Fishtail neighborhood.}
The fishtail neighborhood $F$ is  an elliptic surface which has the handle decomposition in Figure~\ref{F}.
It is well known that the $-1$-framed meridian in Figure~\ref{F} is isotopic to the vanishing cycle of $F$.
In \cite{G2} Gompf proved the following assertion.
\begin{lem}[\cite{G2}]
\label{Gompf}
Let $X$ be a 4-manifold and $T$ be a regular  fiber of a fishtail neighborhood $F$ embedded in $X$.
Then the $q$-fold Dehn twist along $T$ parallel to the vanishing cycle of $F$ does not change 
the diffeomorphism type of $X$.
\end{lem}

%%%%%%%%%%%%%%%%%%%%%%%%%%%%%%%
\begin{figure}[htbp]
\begin{center}
\includegraphics[width=.25\textwidth]{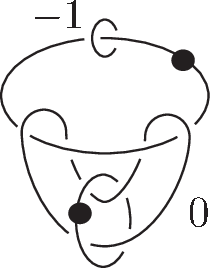}
\caption{A handle decomposition of  $F$.}
\label{F}
\end{center}
\end{figure}
%%%%%%%%%%%%%%%%%%%%%%%%%%%%%%%%

We prove the following.
\begin{lem} \label{lem:handle_decomposition}
The homotopy 4-ball $W_n$ also has the handle decomposition given by the first picture in  Figure \ref{fishtail1}.
\end{lem}

%%%%%%%%%%%%%%%%%%%%%%%%%%
\begin{figure}[htpb]
\begin{center}
\includegraphics[width=1.0\textwidth]{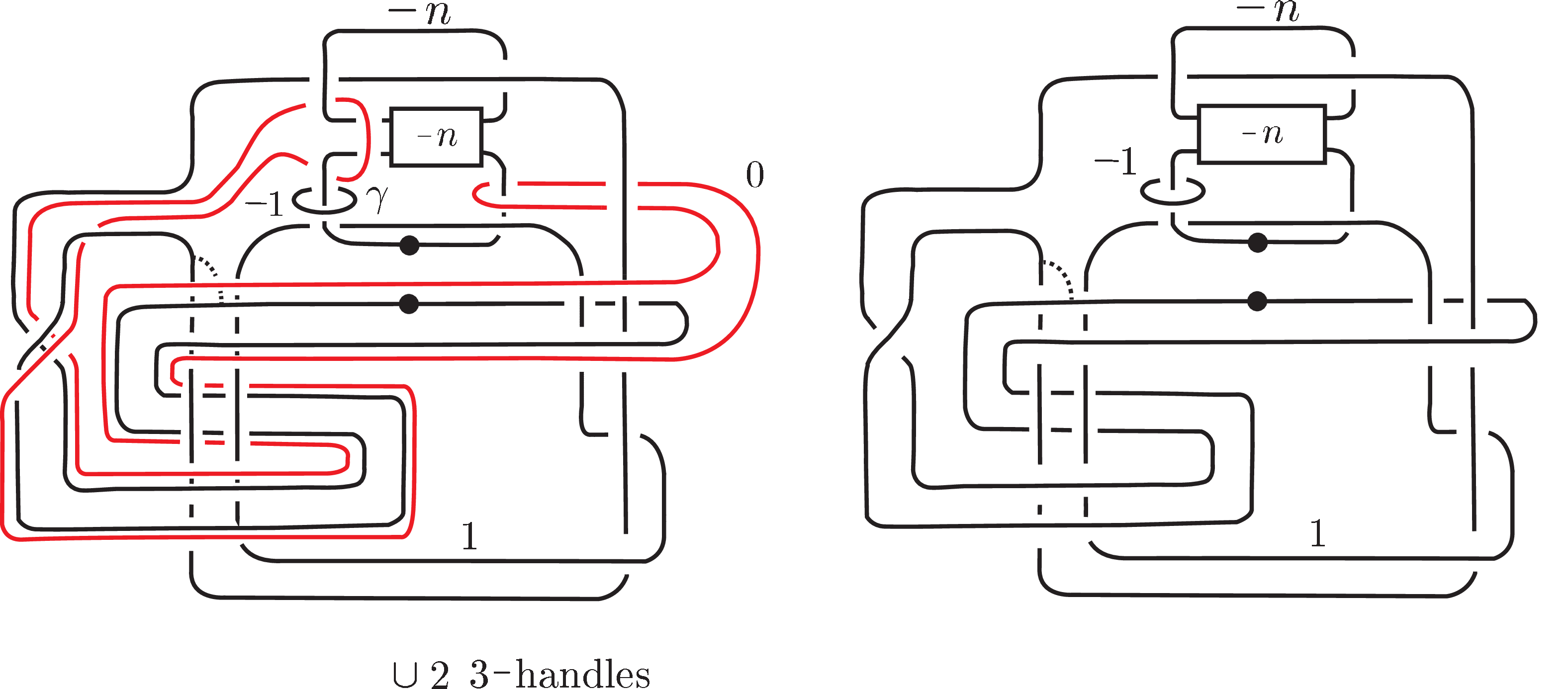}
\caption{A handle decomposition of $W_n$ and the handlebody picture of $W_n + \gamma^{-1}$.}
\label{fishtail1}
\end{center}
\end{figure}
%%%%%%%%%%%%%%%%%%%%%%%%%

\begin{proof} 
We fix a  diffeomorphism identifying $\partial  W_n$
with $S^3$.
We use the diffeomorphism
described in Figure  
\ref{deformation2} again.
Recall that this diffeomorphism tells us that  
the $-1$-framed  $\gamma$  is isotopic to  the 0-framed unknot in  $S^3$
(for the detail, see the proof of Theorem \ref{thm:main1}).
Therefore, by inserting a canceling  2/3-handle pair to $W_n$, 
we obtain 
\[ W_n   \approx W_n + \gamma^{-1} \cup \text{(3-handle)}, \]
where $W_n + \gamma^{-1}$ is the handlebody given by 
the second  picture  in Figure~\ref{fishtail1}.

Next we fix a diffeomorphism identifying $\partial (W_n + \gamma^{-1})$
with $S^1 \times S^2$ described in Figure \ref{fishtail2}
(for a while, we ignore the curve $\mu$).
%%%%%%%%%%%%%%%%%%%%%%%%
\begin{figure}[htpb]
\begin{center}
\includegraphics[width=1.0\textwidth]{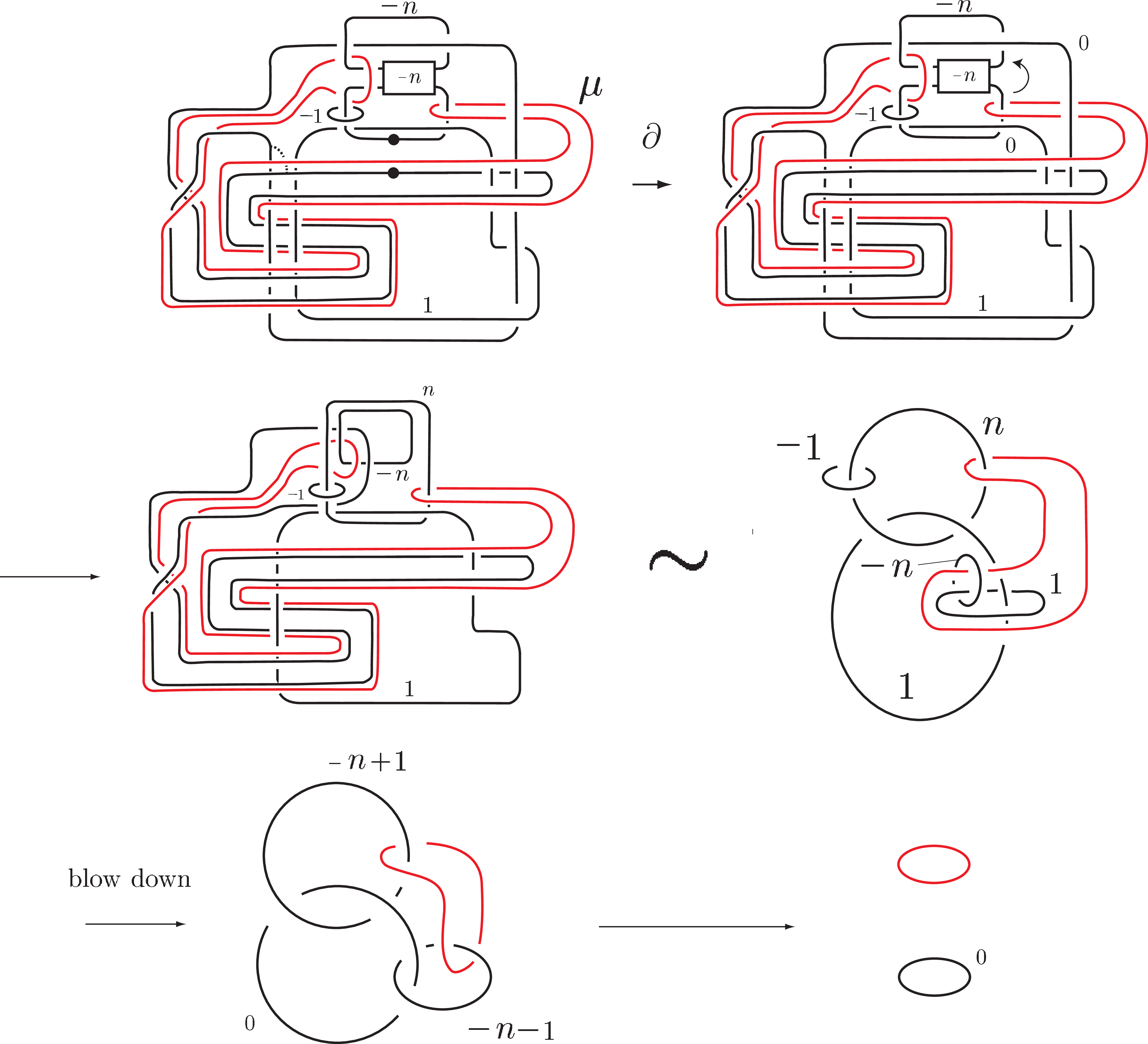}
\caption{A  diffeomorphism identifying $\partial  (W_n + \gamma^{-1})$
with $S^1 \times S^2$ which tells us that 
the curve $\mu$ is the  unknot in $S^1 \times S^2$.}
\label{fishtail2}
\end{center}
\end{figure}
%%%%%%%%%%%%%%%%%%%%%%%%%
This diffeomorphism tells us that $\mu \subset \partial  (W_n + \gamma^{-1})$
%described in   Figure~\ref{fishtail2} 
is the  unknot  in  $S^1 \times S^2$.
Furthermore, if we regard $\mu$ as a $0$-framed knot,
then it is isotopic to the 0-framed unknot in  $S^1 \times S^2$.
Therefore, by inserting a canceling  2/3-handle pair to $W_n$,  
we obtain the first picture in Figure \ref{fishtail1}. 
\end{proof}

Now we prove the main result in this section.

\begin{proof} [Proof of Theorem~\ref{thm:main1} in the case $\mathcal{K}=8_{20}$] 
The second picture of Figure~\ref{embedding} is a sub-handlebody of $W_n$.
By isotopy,
we see that  it  is diffeomorphic to $F\cup \text{(1-handle)}$,
where $F$ is  the fishtail neighborhood.
Therefore, by removing the 1-handle,
we can find $F$ as a submanifold of $W_n$.

%%%%%%%%%%%%%%%%%%%%%%%%%%
\begin{figure}[htpb]
\begin{center}
\includegraphics[width=1.0\textwidth]{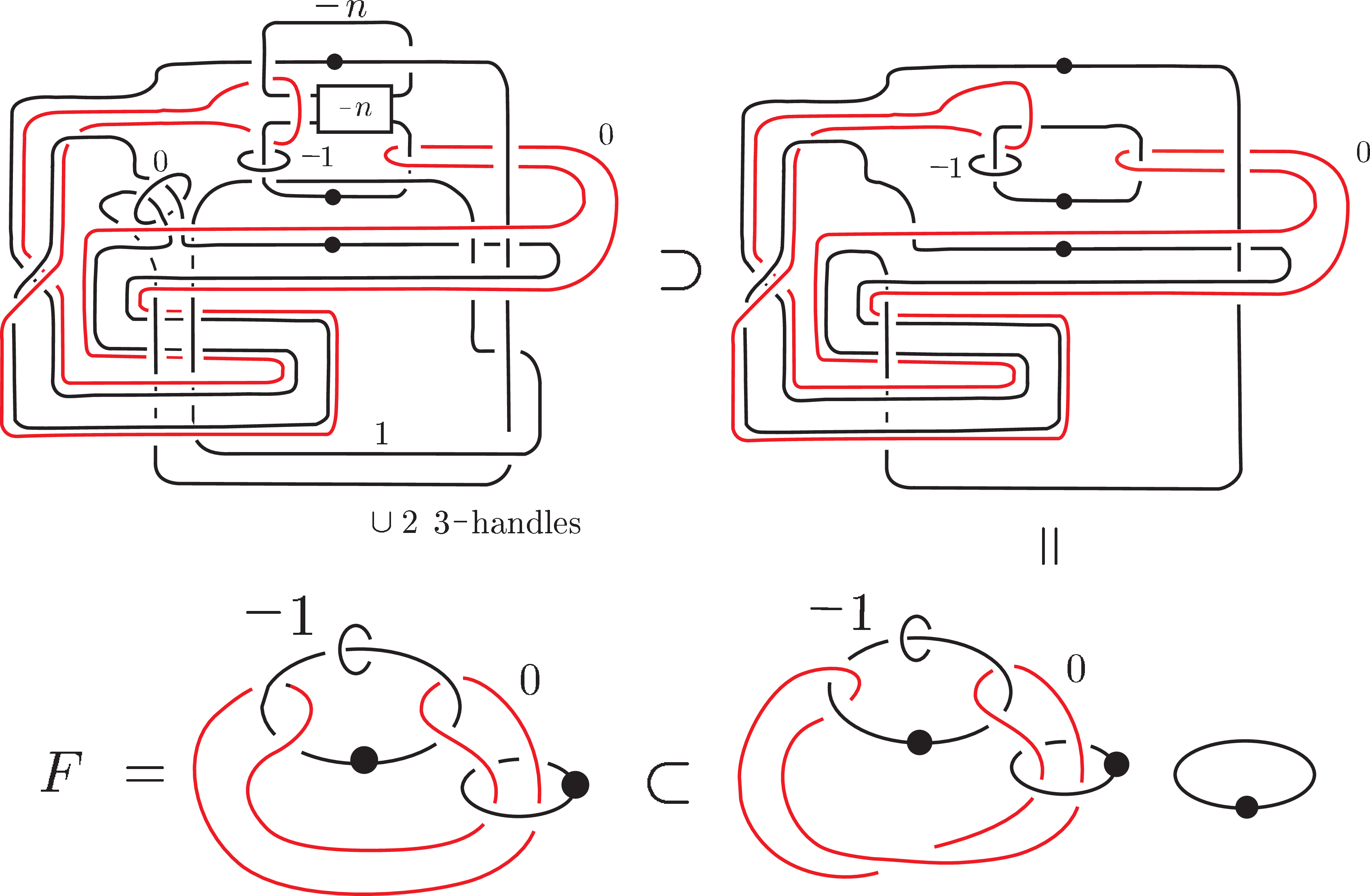}
\caption{An embedding of  the fishtail neighborhood $F$.}
\label{embedding}
\end{center}
\end{figure}
%%%%%%%%%%%%%%%%%%%%%%%%%

Let $T$ be a regular fiber of $F$  embedded in $W_n$.
%and $\gamma$ the closed curve on $T$ which isotopic to the vanishing cycle of $F$.
The 1-fold Dehn twist along $T$ parallel to $\gamma$ is 1-untwisting along $\gamma$.
For the detail, see \cite{AY} or \cite{GS}.
Thus the local deformation is as in Figure~\ref{logtrans}.
As a result,
performing the $n$-fold Dehn twist along $T$ parallel to $\gamma$
and removing the canceling 2/3-handle pairs,
we obtain $W_0$ which is diffeomorphic to $B^4$.  
By Lemma \ref{Gompf},  $W_n \approx W_0$.
Therefore $\mathcal{K}_n$ (obtained from $8_{20}$) is a slice knot.
\end{proof}

%%%%%%%%%%%%%%%%%%%%%%%%%
\begin{figure}[htpb]
\begin{center}
\includegraphics[width=1.0\textwidth]{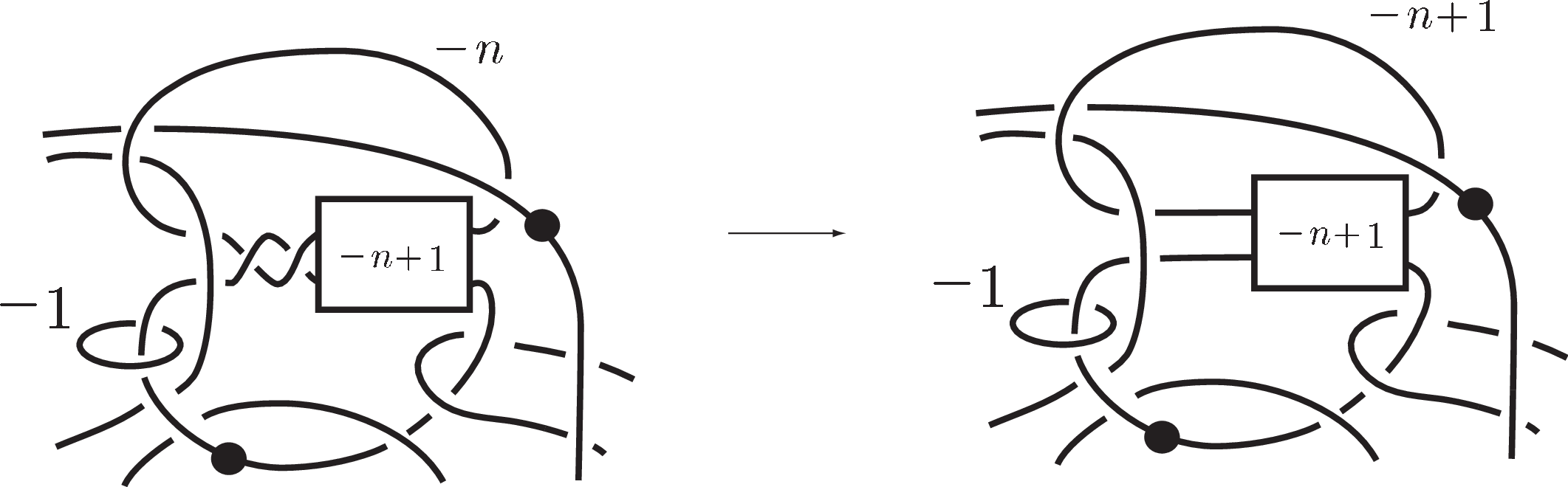}
\caption{The $1$-fold Dehn twist along $T$ parallel to $\gamma$.}
\label{logtrans}
\end{center}
\end{figure}
%%%%%%%%%%%%%%%%%%%%%%%%

\section{A sufficient condition to be ribbon}\label{sec:ribbon} 

In this section, 
we give a sufficient condition for a slice knot
to be ribbon (Lemma \ref{lem:ribbon})
and prove that all the knots obtained from $8_{20}$ by annulus twists are  ribbon (Theorem \ref{thm:ribbon}).

\begin{lem}\label{lem:ribbon}
Let HD be a handle diagram of $B^4$.
Suppose that HD is changed into the empty  handle diagram of $B^4$
by handle slides, adding or canceling 1/2-handle pairs, and isotopies.
Then the belt sphere of  any 2-handle of $HD$ is a ribbon knot.
\end{lem}
\begin{proof}
Let 
$$HD=HD_0\to HD_1\to \cdots \to HD_n=\text{(empty handle diagram) }$$
be a  sequence of handle diagrams 
satisfying the condition of Lemma~\ref{lem:ribbon}.
By rearranging the sequence, 
we can assume  the following.
$$HD_0\to HD_1\to \cdots \to HD_k\ \ (\text{adding canceling 1/2-handle pairs}),$$
$$HD_k\to HD_{k+1}\to \cdots \to HD_l\ \ (\text{handle slides}),$$
$$HD_l\to HD_{l+1}\to \cdots \to HD_n\ \ (\text{annihilating canceling 1/2-handle pairs}).$$

Let  $\beta$ be the belt sphere of any 2-handle of $HD$.
Then it is the unknot in $HD$
and we denote by $\beta_i$ $(i=1, 2, \cdots ,l)$ the corresponding knot in  $HD_i$.
We see that  $\beta_l$  is also the unknot in $HD_{l}$.
Furthermore, 
we can find a smoothly embedded disk $D$ in  $HD_l$ such that $\partial D=\beta_l$,
the disk $D$ does not intersect any dotted 1-handles\footnote{
We can choose $D$ in this way since the link which  consists of dotted circles (representing 1-handles) and $\beta_{l}$ is the unlink.}, and
$D$  intersects transversely with some attaching spheres of 2-handles 
as the left in Figure~\ref{bunkatsu}.
%%%%%%%%%%%%%%%%%%%%%%%%%%%
\begin{figure}[htpb]
\begin{center}
\includegraphics[width=1.0\textwidth]{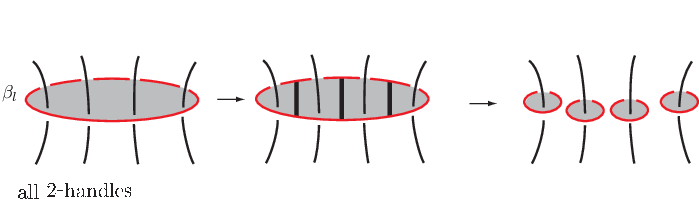}
\caption{Band surgeries along mutually disjoint bands ($m=4$).}
\label{bunkatsu}
\end{center}
\end{figure}
%%%%%%%%%%%%%%%%%%%%%%%%%%%%%%%%
Let $m$ be the number of  intersections between  $D$ and the attaching spheres of 2-handles of  $HD_l$.
By band surgeries along mutually disjoint bands $B_1$, $B_2$, $\cdots$, $B_{m-1}$ 
as the middle picture in Figure~\ref{bunkatsu},
we obtain an $m$-component link $L$ such that 
each component is the  meridian of the attaching sphere of a 2-handle of $HD_l$.
 
%Then, we can separat $D_l$ by properly embeddd arcs in the $D_l$ in the such way that each 
%component by the saparation has one intersection point with the attaching spheres of 2-handles
%As a result, we get several meridians $\bar{\beta}=\bar{\beta}_1\cup\cdots\cup\bar{\beta}_m$ of attaching spheres of the 2-handles,
%because each attaching sphere of the 2-handles has a meridional dotted 1-handle.
%Hence the meridians $\bar{\beta}$ are split to the any component of $HD_l$.

Finally we consider the sequence $HD_l\to \cdots \to HD_n$.
Let $L'$ be the link in $HD_n$ which is  corresponding to $L$.
Then it is the $m$-component unlink in $S^3$.
In other words, the knot $\beta$ is deformed into 
the $m$-component unlink by
band surgeries along $m-1$ bands.
This means that $\beta$ is a ribbon knot.
\end{proof}

Let $8_{20}$ be the knot with 
the  annulus presentation as in the right side of Figure \ref{fig:Def-BP}
and  $\mathcal{K}_n$  $(n \ge 0)$  the knot obtained from $8_{20}$ by the $n$-fold annulus twist.
By Theorem \ref{thm:main1}, 
$\mathcal{K}_n$ is a slice knot. 
There is no apparent reason for $\mathcal{K}_n$ to be ribbon.
Our result is that, indeed, $\mathcal{K}_n$ is a ribbon knot.
To prove this,
we first observe the following.

\begin{lem}\label{Omae}
The slice knot $\mathcal{K}_n$ is located as in  Figure \ref{OmaeKn}.
\end{lem}

\begin{proof}
By the proofs of Lemma \ref{h_slice} and Theorem \ref{thm:main1},
we obtain this lemma immediately.
\end{proof}
%%%%%%%%%%%%%%%%%%%%%%%%%%%%%%%
\begin{figure}[htpb]
\begin{center}
\includegraphics[width=.5\textwidth]{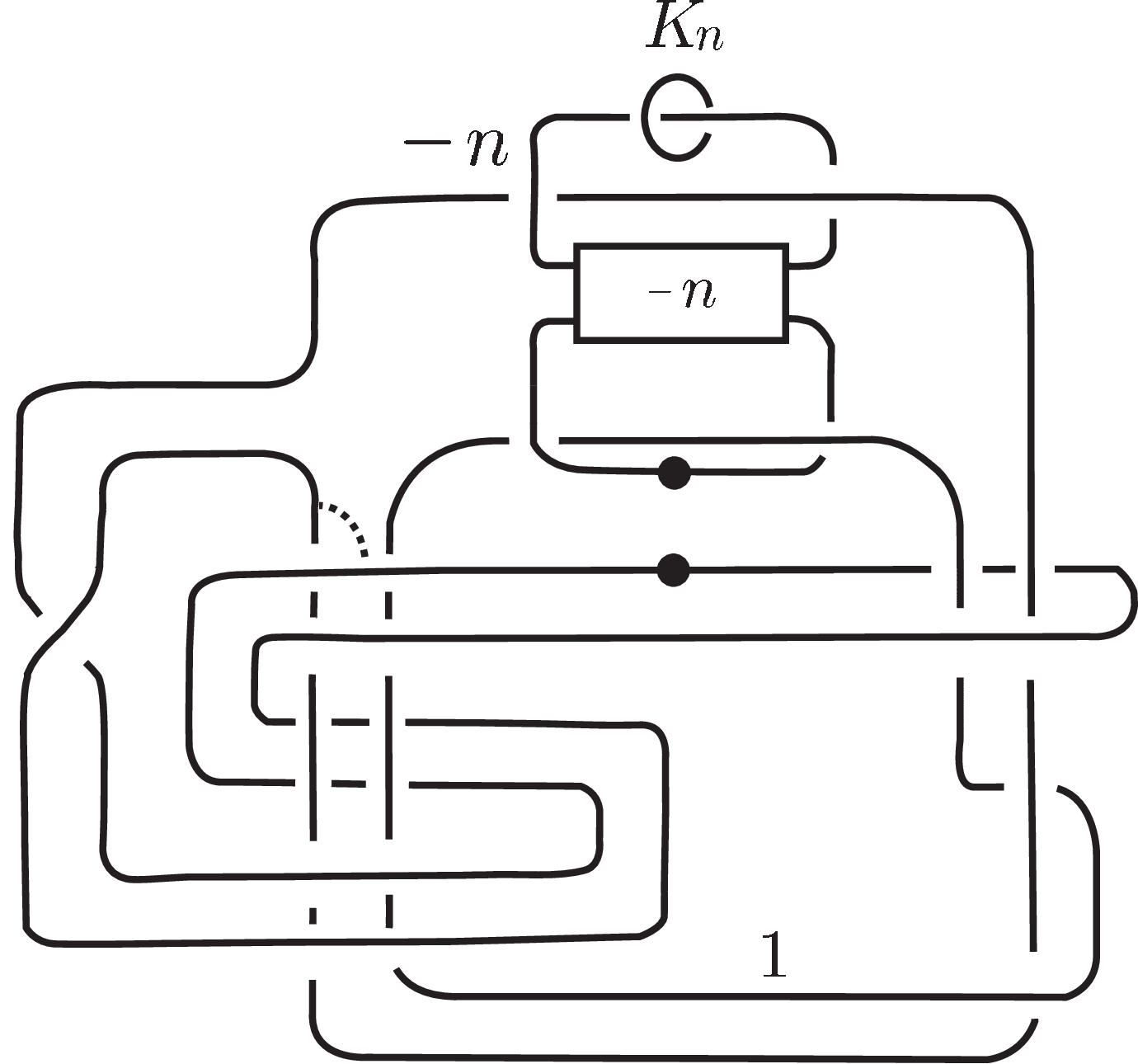}
\caption{The slice knot  $\mathcal{K}_n$ in $\partial W_n$.}
\label{OmaeKn}
\end{center}
\end{figure}
%%%%%%%%%%%%%%%%%%%%%%%%%%%%%%%%%%

\begin{rmk}
Let $K$ be any ribbon knot  in $\partial B^4$.
Then it is not difficult to see that
 $B^4$ admits  a handle decomposition
\[ h^0 \cup h_1^1  \cup \cdots  \cup h_{n}^1 \cup h_{1}^2 \cup \cdots  \cup h_{n}^2 \]
such that the belt sphere of some 2-handle is isotopic to $K$,
where $h^0$ is a $0$-handle, $h_i^1 (i=1, \cdots, n)$ is a 1-handle and
$h_j^2 (j=1, \cdots, n)$ is a 2-handle.
For the converse, see  Conjecture \ref{conj:2}.
\end{rmk}

Now  we prove  the following:

%%%%%%%%%%%%%%%%%%%%%%%%%%%%%%%%
\begin{thm}\label{thm:ribbon}
The  slice knot $\mathcal{K}_n$ $(n \ge 0)$ is ribbon.
\end{thm}

\begin{proof}
Let $HD$ be the handle diagram  given by 
the picture in Figure \ref{OmaeKn}.
By Lemma \ref{Omae},
$\mathcal{K}_n$ is isotopic to the belt sphere of a 2-handle of $HD$.
By Lemma \ref{lem:ribbon},
if $HD$ is changed into the empty handle diagram 
by handle slides, adding or  canceling 1/2-handle pairs, and isotopies,
then $\mathcal{K}_n$ is a ribbon knot.
Such operations are  realized in Figures~\ref{ribbon1},
~\ref{ribbon2},
~\ref{ribbon4} and ~\ref{ribbon3}.
As a result, $\mathcal{K}_n$ is a ribbon knot.
%%%%%%%%%%%%%%%%%%%%%%%%%%%%%%%
\begin{figure}
\begin{center}
\includegraphics[width=1.0\textwidth]{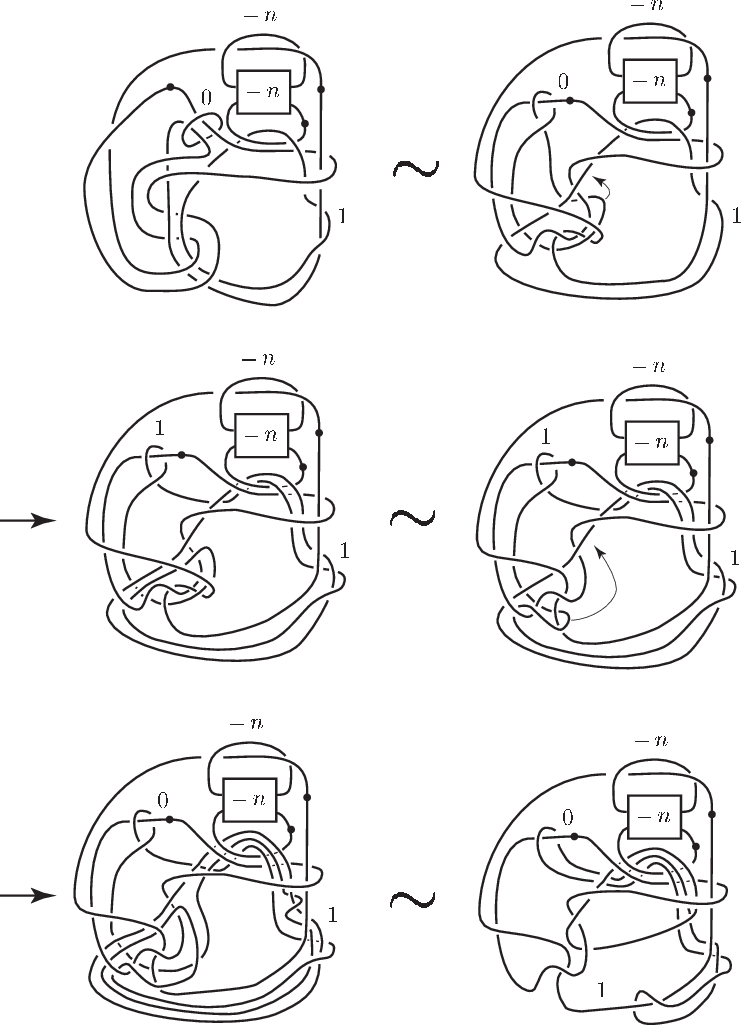}
\caption{Handle calculus without adding canceling 2/3-handle pairs.}
\label{ribbon1}
\end{center}
\end{figure}
\ \ 
\begin{figure}
\begin{center}
\includegraphics[width=1.0\textwidth]{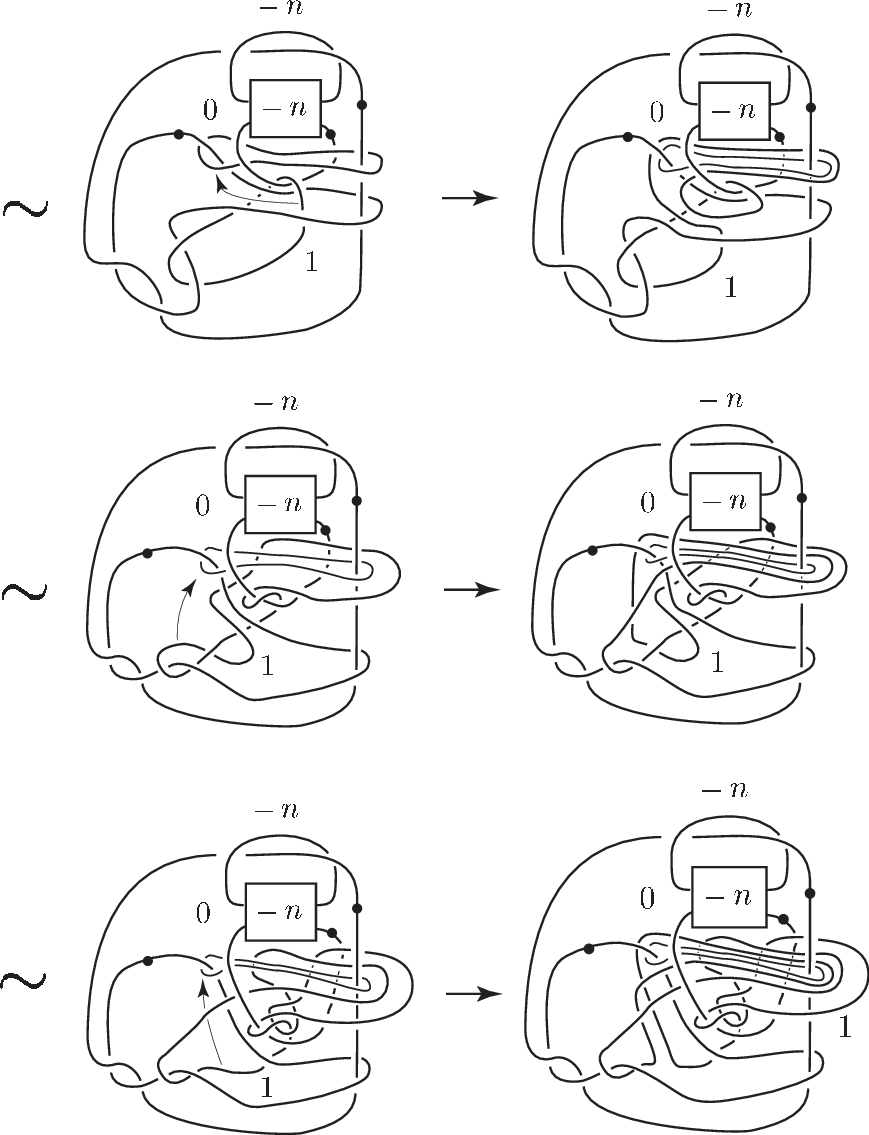}
\caption{Handle calculus without adding canceling 2/3-handle pairs.}
\label{ribbon2}
\end{center}
\end{figure}
\ \
\begin{figure}
\begin{center}
\includegraphics[width=1.0\textwidth]{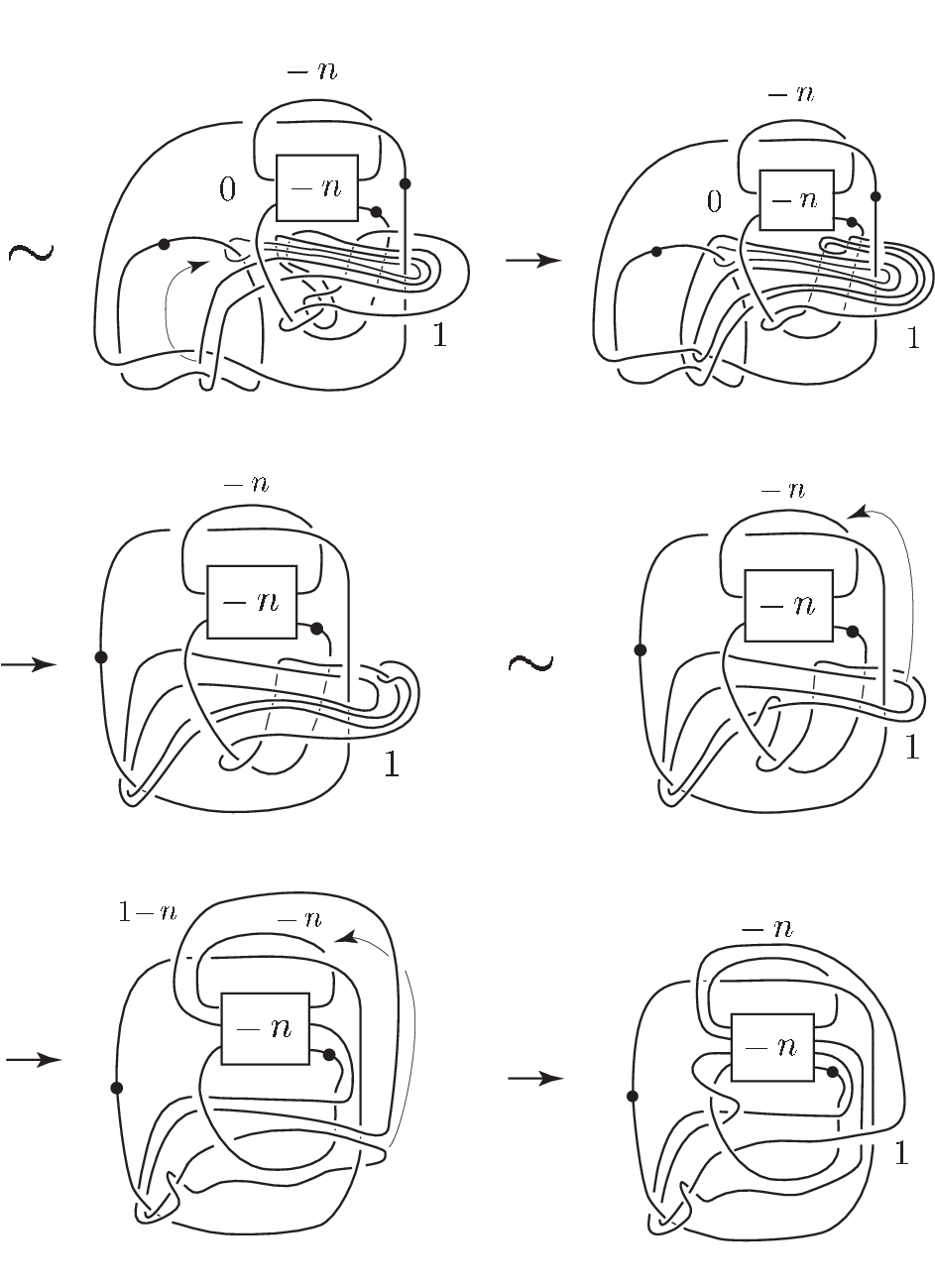}
\caption{Handle calculus without adding canceling 2/3-handle pairs.}
\label{ribbon4}
\end{center}
\end{figure}
\ \ 
\begin{figure}
\begin{center}
\includegraphics[width=1.0\textwidth]{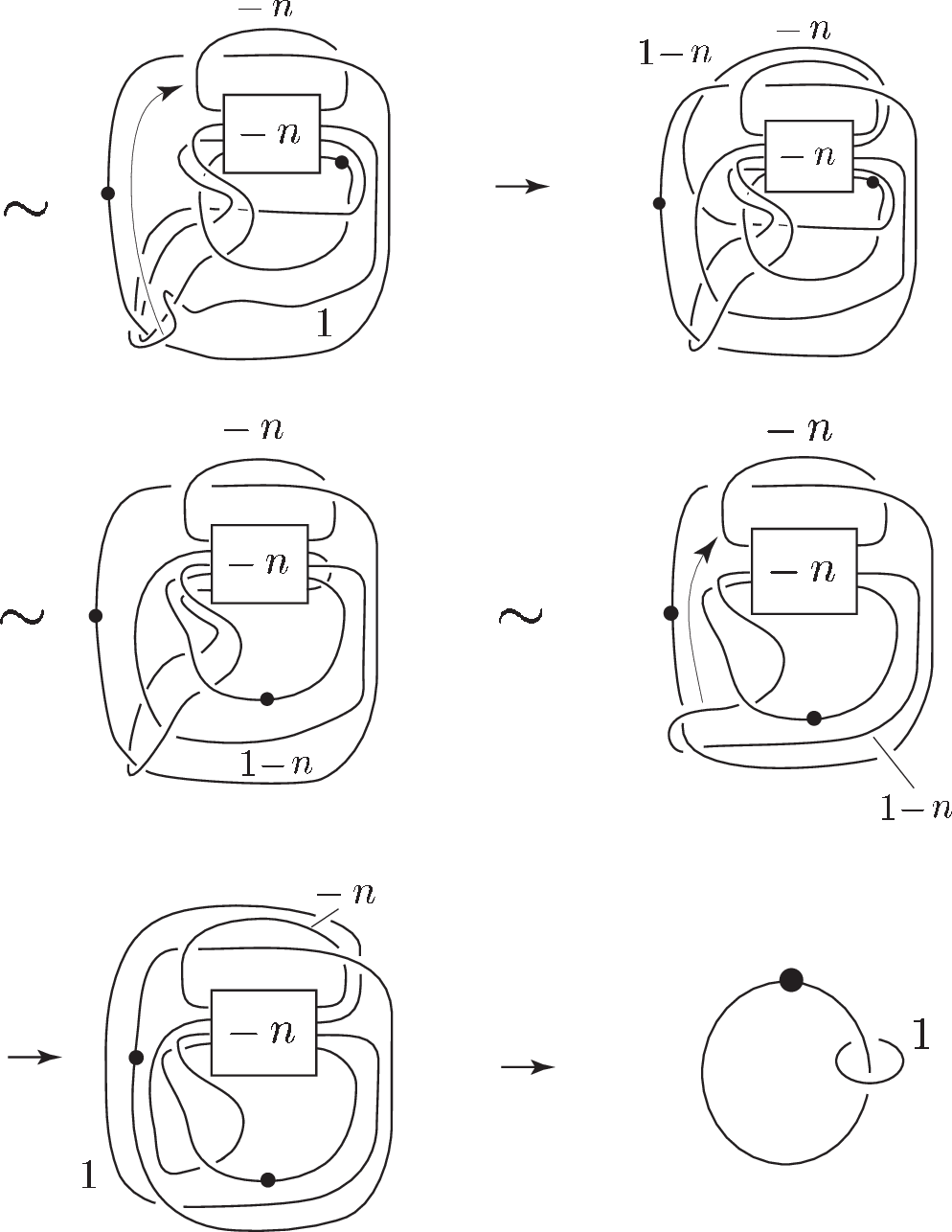}
\caption{Handle calculus without adding canceling 2/3-handle pairs.}
\label{ribbon3}
\end{center}
\end{figure}
%%%%%%%%%%%%%%%%%%%%%%%%%%%%
\end{proof}
%%%%%%%%%%%%%%%%%%%%%%%%%%%

\vskip 30mm
\vskip 30mm
\vskip 30mm
\vskip 30mm
\vskip 30mm

\pagebreak

Now we draw a ribbon presentation of $\mathcal{K}_n$.
Keeping track of  $\mathcal{K}_n$ through the handle calculus,
though it is rather troublesome, 
we can obtain a ribbon presentation of $\mathcal{K}_n$ as in Figure~\ref{ribbonhyoji2}.

\begin{figure}[htpb]
\begin{center}
\includegraphics[width=0.7\textwidth]{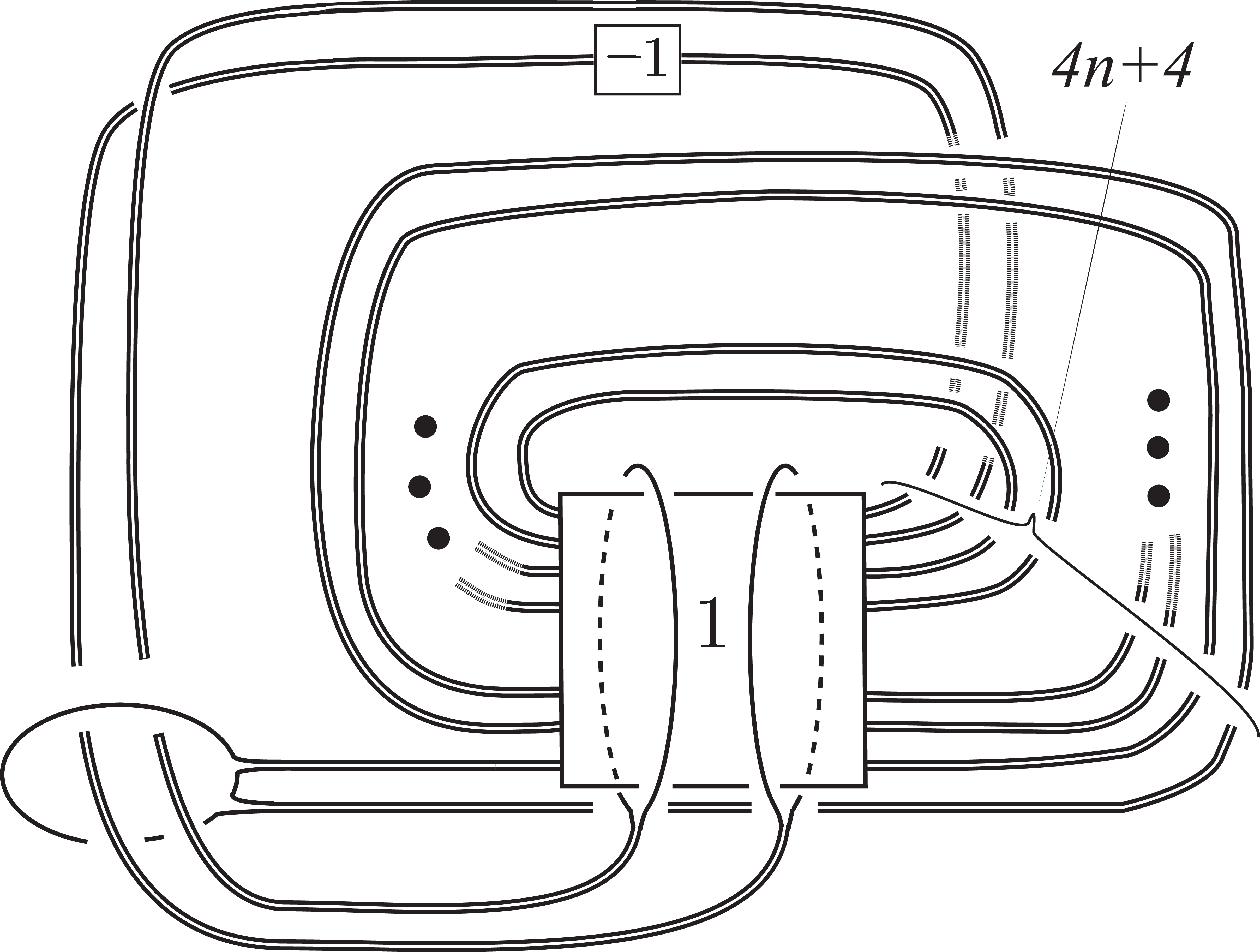}
\caption{A ribbon presentation of $\mathcal{K}_n$ $(n \ge 1)$.}
\label{ribbonhyoji2}
\end{center}
\end{figure}

\section{Two conjectures} \label{sec:Two conjectures}

In this section, we propose two conjectures.
The first one is the following.

\begin{conj} \label{conj:2}
Let $HD$ be a handle diagram  of $B^4$  without 3-handles.
Then the belt-sphere of any $2$-handle of $HD$
is a ribbon knot.
\end{conj}

Recall that  each slice knot in Theorem  \ref{thm:main1}  is isotopic to  the belt-sphere
of  a $2$-handle of  a certain handle diagram of $B^4$  without 3-handles,
see the proofs of Lemma  \ref{h_slice}  and Theorem \ref{thm:main1}.
Therefore, if  Conjecture \ref{conj:2}  is true,
then  all slice knots  in Theorem  \ref{thm:main1} are ribbon.

A partial answer to Conjecture \ref{conj:2} has  already given in Lemma \ref{lem:ribbon}. 
The difficulty to solve  this conjecture   is explained by yet another  following conjecture.

\begin{conj} \label{conj:1}
There exists a   handle diagram  $HD$ of  $B^4$  without 3-handles
such that we always have to add  canceling 2/3-handle pairs
when we change $HD$  into the empty  handle diagram $B^4$ by a sequence of 
handle slides, adding or canceling handle pairs, and isotopies
\end{conj}

A promising candidate  to Conjecture \ref{conj:1} is the handle diagram  $H_{n,k}$ of $B^4$
given by Gompf \cite{G1} (see the left half of Figure \ref{Hnk}),
where $n \ge 3$ and $k \neq 0$.
\begin{figure}[htpb]
\begin{center}
\includegraphics[width=1.0\textwidth]{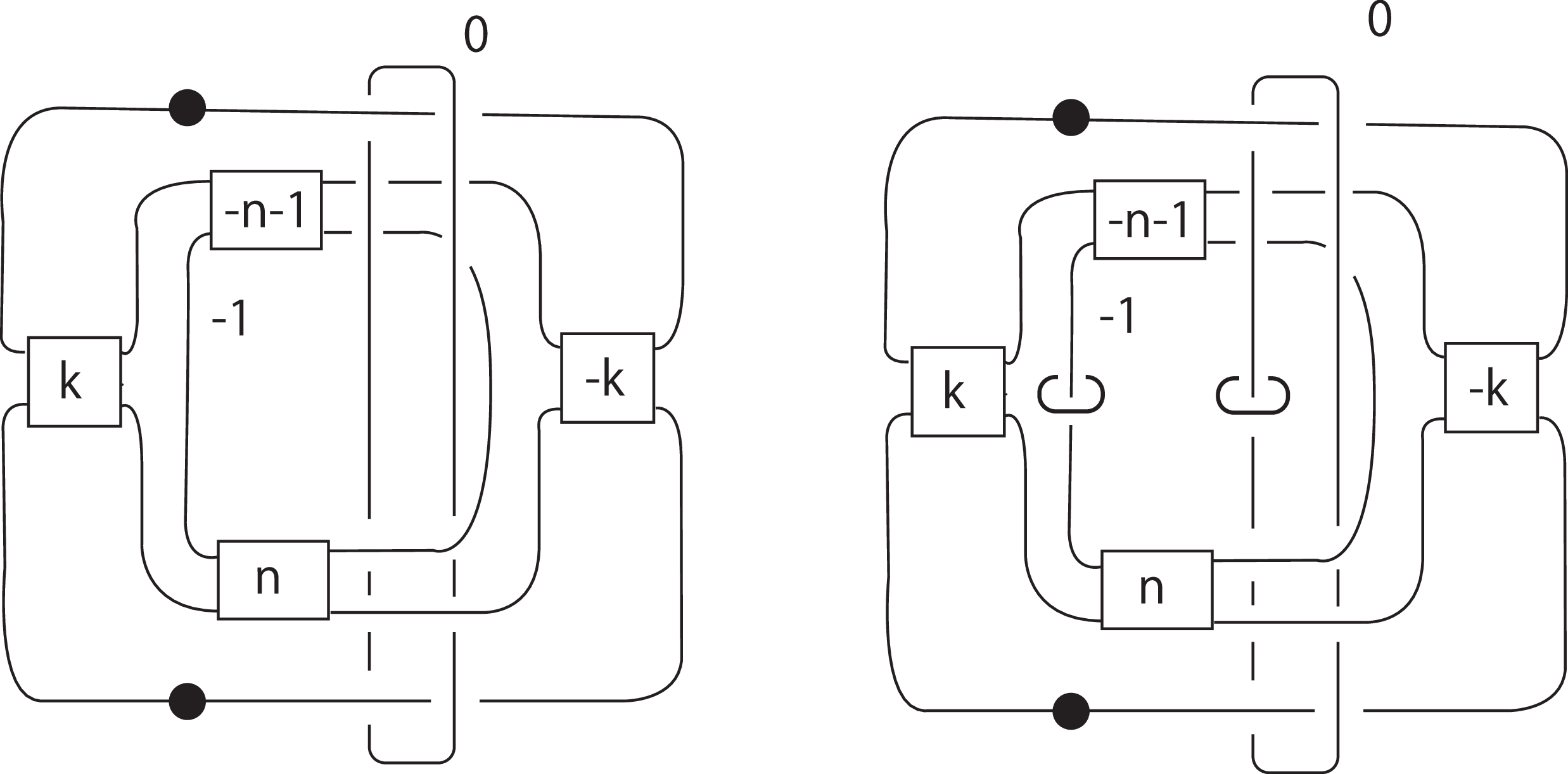}
\caption{The handle diagram $H_{n,k}$ of $B^4$ (left) and the
$2$-component link $L_{n,k}$ in $S^3=\partial B^4$ (right).}
\label{Hnk}
\end{center}
\end{figure}

Finally, we observe that, if Conjecture \ref{conj:2}  is true,
then  Gompf, Scharlemann and Thompson's slice knots in \cite{GST} are ribbon
as follows: 
Let $L_{n,k}$ be the $2$-component link in $S^3$ which 
consists of the two belt-spheres of the two $2$-handles of $H_{n,k}$,
see the right half of Figure \ref{Hnk}.
By the definition,
 $L_{n,k}$ is a slice link, that is, it bounds two smoothly embedded disjoint disks in $B^4$.
Each Gompf, Scharlemann and Thompson's slice knot  is 
obtained from $L_{n, k}$ by attaching an arbitrary band.
After a single $2$-handle slide (along the band),
it turns out that 
the slice knot 
is  isotopic to the belt-sphere
of  a $2$-handle of  a certain handle diagram  of $B^4$  without 3-handles.
Therefore, if  Conjecture \ref{conj:2}  is true,
 these  slice knots are also ribbon.
In this sense,  to solve Conjecture \ref{conj:2} is the first step toward an
affirmative answer  to the slice-ribbon conjecture.

\end{document}